\documentclass[12pt]{article}
\usepackage{makeidx}
\usepackage{amsfonts}
\usepackage{amssymb,indentfirst}
\usepackage{amsmath}
\usepackage{lineno}
\usepackage[numbers,sort&compress]{natbib}
\usepackage[colorlinks,linkcolor=blue,citecolor=blue,urlcolor=blue]{hyperref}
\usepackage{graphicx,indentfirst,tabularx}

\newtheorem{theorem}{Theorem}

\newtheorem{definition}[theorem]{Definition}
\newtheorem{example}[theorem]{Example}

\newtheorem{lemma}[theorem]{Lemma}

\newtheorem{remark}[theorem]{Remark}

\newenvironment{proof}[1][Proof]{\noindent\textbf{#1.} }{\ \rule{0.5em}{0.5em}}
\input{tcilatex}
\begin{document}

\title{The Taylor Integral and a Generalization of the Discrete Fourier
Transform}
\author{Athanasios Christou Micheas \thanks{%
Department of Statistics, University of Missouri, 146 Middlebush Hall,
Columbia, MO 65211-6100, USA, email: micheasa@missouri.edu}}
\maketitle

\begin{abstract}
We propose a new integral based on Taylor measures, study its properties
extensively, and we illustrate that it includes many concepts from
mathematics as special cases. In particular, the new integral emerges as a
generalization of the discrete Fourier transform, and we identify general
conditions for it to be invertible when applied to any real or complex
sequence. Applications to the mathematical sciences are also presented.
\end{abstract}



\section{Introduction}

The concept of a Taylor measure was defined and studied in \cite%
{micheas2025taylor}. We denote the collection of all signed, finite Taylor
measures by%
\begin{equation}
\mathcal{T}^{\mathcal{F}}=\left\{ T_{\gamma ,\mathbf{a}}:T_{\gamma ,\mathbf{a%
}}(B)=\sum\limits_{n\in B}a_{n}\frac{\gamma ^{n}}{n!}\text{, }B\in \mathcal{B%
}(\mathbb{N}),\text{ }a_{n},\gamma \in \mathbb{M},\text{ with }T_{\gamma ,%
\mathbf{a}}(\mathbb{N})<+\infty \right\} .  \label{FiniteTaylorMeasureSpace}
\end{equation}%
where $a_{n},\gamma $ satisfy one of the conditions of Theorem 1 in \cite%
{micheas2025taylor} to guarantee finiteness, and we write $\mathbb{M}$ to
denote the real numbers $\Re $ or the complex numbers $\mathbb{C}$, and $%
\mathcal{B}(\mathbb{M})$ to denote the Borel sets over $\mathbb{M},$ defined
via the usual norms. Properties of the space $\mathcal{T}^{\mathcal{F}}$ and
first applications were presented in \cite{micheas2025taylor}. In \cite%
{micheas2026multivariate}, theory and applications of the Taylor measure
were presented to functional analysis.

In this paper we exemplify Taylor measures by defining and investigating
integrals based on integrators from $\mathcal{T}^{\mathcal{F}},$ and we use
the generic term Taylor integral (TI) to describe all the constructions
presented below. The TI is shown to be a unifying framework that contains
many important concepts from mathematics as special cases, including
Dirichlet's $\eta $ function, discrete Fourier transform, Euler's totient
function, generating functions, hypergeometric series, Riemann's $\zeta $
function, and more. Furthermore, even though most of the exposition in the
paper involves real numbers it is straightforward to give all the
definitions, proofs and examples, in terms of complex numbers.

The paper proceeds as follows; in Section 2 we introduce various forms of
the Taylor integral and study its properties extensively. A characterization
of absolutely convergent sequences via TIs is also provided. In Section 3 we
introduce and study the discrete Taylor transformation (DTT), a special case
of the TI, as well as the inverse discrete Taylor transformation, and give
examples in different mathematical contexts. This TI emerges as a unifying
framework of several mathematical tools, such as the discrete Fourier
transform, and ordinary or exponential generating functions. Importantly, we
prove that the discrete Taylor transformation can uniquely identify $\mathbb{%
M}$-valued sequences and in addition, we identify the conditions required
for the DTT to be invertible. An alternative method to Vandermonde matrix
inversion is also provided, via Monte Carlo. In Section 4, we illustrate how
to use the TI to study significant applications in mathematics, including
cryptography and computer graphics. Concluding remarks are given in the last
section.

\section{Construction of the Taylor Integral}

We explore several constructions of the Taylor integral, including
integration with respect to specific choices of signed Taylor measure, that
yield important special cases. In what follows, we will assume that the
integrands are real or complex sequences $s_{n},$ $n\in \mathbb{N},$ and we
write $\mu _{p}$ to denote Lebesgue measure defined over the measurable
space $(\mathbb{M},\mathcal{B}(\mathbb{M}))$. We will write $|z|$ to denote
absolute value if $z\in \Re $ or the complex modulus ($|z|=z\overline{z},$
where $\overline{z}$ denotes the conjugate of the complex number $z)$ when $%
z\in \mathbb{C}.$

Take $\mathbb{M}=\mathbb{R},$ and consider any finite, signed Taylor measure
$T_{\gamma ,\mathbf{a}}\in \mathcal{T}^{\mathcal{F}},$ $\gamma \in \Re ,$ $%
\gamma \neq 0,$ $\mathbf{a}=$ $(a_{0},a_{1},...)\in \Re ^{\infty },$ and any
real sequence $s_{n}.$ Using Jordan decomposition theorem (e.g., \cite%
{micheas2018theory}, Theorem 3.8) for the signed measure $T_{\gamma ,\mathbf{%
a}}$, there exist two mutually singular measures denoted by $T_{\gamma ,%
\mathbf{a}}^{+}$ and $T_{\gamma ,\mathbf{a}}^{-}$ such that%
\begin{equation}
T_{\gamma ,\mathbf{a}}=T_{\gamma ,\mathbf{a}}^{+}-T_{\gamma ,\mathbf{a}}^{-},
\label{SignedPlusMinus}
\end{equation}%
with this decomposition being unique. More precisely, if $\left\{
A^{+},A^{-}\right\} $ is a Hahn decomposition (e.g., \cite{micheas2018theory}%
, Theorem 3.7) of $T_{\gamma ,\mathbf{a}},$ we have%
\begin{eqnarray*}
T_{\gamma ,\mathbf{a}}^{+}(B) &=&T_{\gamma ,\mathbf{a}}(B\cap A^{+}),\text{
and} \\
T_{\gamma ,\mathbf{a}}^{-}(B) &=&-T_{\gamma ,\mathbf{a}}(B\cap A^{-}),
\end{eqnarray*}%
with $0\leq T_{\gamma ,\mathbf{a}}^{+}(B),T_{\gamma ,\mathbf{a}%
}^{-}(B)<+\infty ,$ $\forall B\in \mathcal{B}(\mathbb{N}).$ We begin with
the definition of the Taylor integral with respect to a Taylor measure.

\begin{definition}[Taylor Integral]
\label{TaylorMeasureIntegral}Let $T_{\gamma ,\mathbf{a}}\in \mathcal{T}^{%
\mathcal{F}}$, and consider a sequence $s_{n}:\mathbb{N}\rightarrow \Re ,$
with $\mathbf{s}=[s_{0},$ $s_{1},$ $...].$ The positive Taylor integral is
defined by
\begin{equation}
P_{\gamma ,\mathbf{a}}^{\mathbf{s},B}=\int\limits_{B}s_{n}T_{\gamma ,\mathbf{%
a}}^{+}(dn):=\sum\limits_{n\in B}s_{n}T_{\gamma ,\mathbf{a}}^{+}(\{n\}),
\label{PosInt1}
\end{equation}%
and the negative Taylor integral is defined by%
\begin{equation}
N_{\gamma ,\mathbf{a}}^{\mathbf{s},B}=\int\limits_{B}s_{n}T_{\gamma ,\mathbf{%
a}}^{-}(dn):=\sum\limits_{n\in B}s_{n}T_{\gamma ,\mathbf{a}}^{-}(\{n\}),
\label{NegInt1}
\end{equation}%
for all $B\in \mathcal{B}(\mathbb{N}),$ with $T_{\gamma ,\mathbf{a}%
}^{+},T_{\gamma ,\mathbf{a}}^{-}$ satisfying Equation (\ref{SignedPlusMinus}%
). The Taylor integral of $\mathbf{s}$\ with respect to $T_{\gamma ,\mathbf{a%
}}$ is defined by%
\begin{equation}
I_{\gamma ,\mathbf{a}}^{\mathbf{s},B}=\int\limits_{B}s_{n}T_{\gamma ,\mathbf{%
a}}(dn):=P_{\gamma ,\mathbf{a}}^{\mathbf{s},B}-N_{\gamma ,\mathbf{a}}^{%
\mathbf{s},B}=\sum\limits_{n\in B}s_{n}\left[ T_{\gamma ,\mathbf{a}%
}^{+}(\{n\})-T_{\gamma ,\mathbf{a}}^{-}(\{n\})\right] =\sum\limits_{n\in
B}s_{n}T_{\gamma ,\mathbf{a}}(\{n\}),  \label{Int1signedmeas}
\end{equation}%
and we say that the integral $I_{\gamma ,\mathbf{a}}^{\mathbf{s},B}$ exists
when one of the integrals $P_{\gamma ,\mathbf{a}}^{\mathbf{s},B}$ or $%
N_{\gamma ,\mathbf{a}}^{\mathbf{s},B}$ is finite, or when both are infinite
they are such that we do not have the indeterminate form $\infty -\infty $.
The sequence $\mathbf{s}$ is called integrable with respect to $T_{\gamma ,%
\mathbf{a}},$ if%
\begin{equation*}
\int\limits_{\mathbb{N}}|s_{n}|\left\Vert T_{\gamma ,\mathbf{a}}\right\Vert
(dn)<+\infty ,
\end{equation*}%
where $\left\Vert T_{\gamma ,\mathbf{a}}\right\Vert =T_{\gamma ,\mathbf{a}%
}^{+}+T_{\gamma ,\mathbf{a}}^{-},$ the total variation, and we write $%
\mathbf{s}$ is integrable $[T_{\gamma ,\mathbf{a}}]$ (with respect to $%
T_{\gamma ,\mathbf{a}}$).\newline
When $B$ is countably infinite, the sum of Equation (\ref{Int1signedmeas})
converges under similar assumptions required for the Taylor measure $%
T_{\gamma ,\mathbf{a}}(B)$ to be finite; $I_{\gamma ,\mathbf{a}}^{\mathbf{s}%
,B}$ is finite provided that one of the following conditions hold:\newline
a) $s_{n}a_{n}$ are uniformly bounded, i.e., $\left\vert
s_{n}a_{n}\right\vert \leq M,$ $\forall n\in \mathbb{N}$, for some $M>0,$ or%
\newline
b) $s_{n}a_{n}$ are asymptotically equivalent to $Ma^{n}$, i.e., $%
s_{n}a_{n}\thicksim Ma^{n},$ for some $M,a\in \Re $, or\newline
c) if $s_{n}a_{n}\neq 0,$ $\forall n>n_{0}\in \mathbb{N},$ then $\underset{%
n\rightarrow +\infty }{\lim }\frac{\left\vert \gamma
s_{n+1}a_{n+1}\right\vert }{(n+1)\left\vert s_{n}a_{n}\right\vert }<1,$
guarantees absolute convergence (and therefore convergence).
\end{definition}

The Taylor integral of $\mathbf{s}$\ with respect to $T_{\gamma ,\mathbf{a}%
}, $ as defined by (\ref{Int1signedmeas}), satisfies%
\begin{equation}
I_{\gamma ,\mathbf{a}}^{\mathbf{s},B}=\sum\limits_{n\in B}s_{n}T_{\gamma ,%
\mathbf{a}}(\{n\})=\sum\limits_{n\in B}s_{n}a_{n}\frac{\gamma ^{n}}{n!}%
=T_{\gamma ,\mathbf{s}\cdot \mathbf{a}}(B)=I_{\gamma ,\mathbf{s}\cdot
\mathbf{a}}^{\mathbf{1},B},  \label{Int1}
\end{equation}%
for all $B\in \mathcal{B}(\mathbb{N}),$ where $\mathbf{s}\cdot \mathbf{a}%
=[s_{0}a_{0},s_{1}a_{1},...]$ is interpreted as dimension-wise
multiplication of the two sequences $\mathbf{s}$ and $\mathbf{a},$ and $%
\mathbf{1}=[1,1,...]$ denotes the unit vector$.$ We write $\left\vert
\mathbf{s}\right\vert $ for the sequence of absolute values $(\left\vert
s_{0}\right\vert ,\left\vert s_{1}\right\vert ,...).$

For the case $\mathbb{M}=\mathbb{C},$ a complex measure does not generally
admit a Jordan decomposition into two positive measures, so instead, one
would write first
\begin{equation*}
T_{\gamma ,\mathbf{a}}=\func{Re}T_{\gamma ,\mathbf{a}}+i\func{Im}T_{\gamma ,%
\mathbf{a}},
\end{equation*}%
where $\func{Re}T_{\gamma ,\mathbf{a}}$ and $\func{Im}T_{\gamma ,\mathbf{a}}$
are real signed measures, and then define the TI by%
\begin{equation*}
I_{\gamma ,\mathbf{a}}^{\mathbf{s},B}=\int\limits_{B}s_{n}dT_{\gamma ,%
\mathbf{a}}=\int\limits_{B}s_{n}d\left( \func{Re}T_{\gamma ,\mathbf{a}%
}\right) +i\int\limits_{B}s_{n}d\left( \func{Im}T_{\gamma ,\mathbf{a}%
}\right) .
\end{equation*}%
or directly, using the aforementioned discrete setting, by
\begin{equation}
I_{\gamma ,\mathbf{a}}^{\mathbf{s},B}=\int\limits_{B}s_{n}dT_{\gamma ,%
\mathbf{a}}=\sum_{n\in B}s_{n}a_{n}\frac{\gamma ^{n}}{n!},  \label{TIform}
\end{equation}%
provided
\begin{equation*}
\sum_{n\in B}|s_{n}||a_{n}|\frac{|\gamma |^{n}}{n!}<\infty ,
\end{equation*}%
where $\left\vert \cdot \right\vert $ denotes modulus in this case.

\subsection{Properties of the integral}

Note that since $T_{\gamma ,\mathbf{a}}$ is a signed measure, it does not
have all the properties we are familiar with from standard integral
definitions (see \cite{micheas2018theory}, Remark 3.16), e.g., the usual
ordering (monotonicity) property does not hold in general, and as a result,
standard convergence results such as Monotone Convergence Theorem (MCT) or
Fatou's Lemma are invalid. However, important properties such as linearity
still hold and can be employed as needed. We collect some of these
properties of the new integral next.

\begin{remark}[Properties of $I_{\protect\gamma ,\mathbf{a}}^{s,B}$]
\label{PropsTMI}We present some of the important properties of the Taylor
integral. In what follows, take any sequences $s_{n},f_{n},g_{n}:$ $\mathbb{N%
}\rightarrow \Re ,$ with $\mathbf{s}=(s_{0},s_{1},...),$ $\mathbf{f}%
=(f_{0},f_{1},...)$ and $\mathbf{g}=(g_{0},g_{1},...).$

\begin{enumerate}
\item No Monotonicity: Note that this property is violated, i.e., if $0\leq
f_{n}\leq g_{n},$ $\forall n\in \mathbb{N},$ it does not necessarily mean
that $I_{\gamma ,\mathbf{a}}^{\mathbf{f},B}\leq I_{\gamma ,\mathbf{a}}^{%
\mathbf{g},B}.$

\item Linearity: For any real constants $\lambda ,\xi \in \Re ,$ it is easy
to see that%
\begin{eqnarray*}
I_{\gamma ,\mathbf{a}}^{\lambda \mathbf{f}+\xi \mathbf{g},B}
&=&\int\limits_{B}(\lambda f_{n}+\xi g_{n})T_{\gamma ,\mathbf{a}%
}(dn)=\sum\limits_{n\in B}(\lambda f_{n}+\xi g_{n})a_{n}\frac{\gamma ^{n}}{n!%
} \\
&=&\lambda \sum\limits_{n\in B}f_{n}a_{n}\frac{\gamma ^{n}}{n!}+\xi
\sum\limits_{n\in B}g_{n}a_{n}\frac{\gamma ^{n}}{n!}=\lambda I_{\gamma ,%
\mathbf{a}}^{\mathbf{f},B}+\xi I_{\gamma ,\mathbf{a}}^{\mathbf{g},B},
\end{eqnarray*}%
and therefore we have linearity of integral.

\item Additivity: For any disjoint sets $A,B\in \mathcal{B}(\mathbb{N}),$ we
have by definition $I_{\gamma ,\mathbf{a}}^{\mathbf{s},A\cup B}=I_{\gamma ,%
\mathbf{a}}^{\mathbf{s},A}+I_{\gamma ,\mathbf{a}}^{\mathbf{s},B}.$

\item Triangle Inequality and Bounds: If $\int\limits_{B}s_{n}T_{\gamma ,%
\mathbf{a}}(dn)$ is defined, then%
\begin{eqnarray*}
\left\vert \int\limits_{B}s_{n}T_{\gamma ,\mathbf{a}}(dn)\right\vert
&=&\left\vert P_{\gamma ,\mathbf{a}}^{\mathbf{s},B}-N_{\gamma ,\mathbf{a}}^{%
\mathbf{s},B}\right\vert <\left\vert P_{\gamma ,\mathbf{a}}^{\mathbf{s}%
,B}\right\vert +\left\vert N_{\gamma ,\mathbf{a}}^{\mathbf{s},B}\right\vert
\\
&<&\sum\limits_{n\in B}|s_{n}|T_{\gamma ,\mathbf{a}}^{+}(\{n\})+\sum%
\limits_{n\in B}|s_{n}|T_{\gamma ,\mathbf{a}}^{-}(\{n\})=\int%
\limits_{B}|s_{n}|\left\Vert T_{\gamma ,\mathbf{a}}\right\Vert (dn),
\end{eqnarray*}%
and therefore,%
\begin{equation*}
\left\vert \int\limits_{B}s_{n}T_{\gamma ,\mathbf{a}}(dn)\right\vert \leq
\int\limits_{B}|s_{n}|\left\Vert T_{\gamma ,\mathbf{a}}\right\Vert (dn),
\end{equation*}%
for any $B\in \mathcal{B}(\mathbb{N})$. The usual bounds for the absolute
value of $I_{\gamma ,\mathbf{a}}^{\mathbf{s},B}$ are given by%
\begin{equation*}
\left\vert \left\vert P_{\gamma ,\mathbf{a}}^{\mathbf{s},B}\right\vert
-\left\vert N_{\gamma ,\mathbf{a}}^{\mathbf{s},B}\right\vert \right\vert
\leq \left\vert I_{\gamma ,\mathbf{a}}^{\mathbf{s},B}\right\vert \leq
\left\vert P_{\gamma ,\mathbf{a}}^{\mathbf{s},B}\right\vert +\left\vert
N_{\gamma ,\mathbf{a}}^{\mathbf{s},B}\right\vert .
\end{equation*}

\item Almost Everywhere: For any set $A\in \mathcal{B}(\mathbb{N}),$ if $%
\left\Vert T_{\gamma ,\mathbf{a}}\right\Vert (A^{c})=T_{\gamma ,\mathbf{a}%
}^{+}(A^{c})+T_{\gamma ,\mathbf{a}}^{-}(A^{c})=0,$ we say that $T_{\gamma ,%
\mathbf{a}}$ vanishes over $A^{c}$ or that $A$ a.e. $[\left\Vert T_{\gamma ,%
\mathbf{a}}\right\Vert ]$ ($A$ almost everywhere with respect to $\left\Vert
T_{\gamma ,\mathbf{a}}\right\Vert )$. Here one needs to use total variation,
otherwise the measure space would not be complete, i.e., the requirement $%
T_{\gamma ,\mathbf{a}}(A^{c})=T_{\gamma ,\mathbf{a}}^{+}(A^{c})-T_{\gamma ,%
\mathbf{a}}^{-}(A^{c})=0,$ for $A$ a.e. $[T_{\gamma ,\mathbf{a}}]$ is not
appropriate, since cancellation can make the signed mass of a set equal to
zero even though the measure is nonzero on subsets of that set.

\item Zero sequence: If $s_{n}=0$ a.e. $[\left\Vert T_{\gamma ,\mathbf{a}%
}\right\Vert ],$ i.e., $\left\Vert T_{\gamma ,\mathbf{a}}\right\Vert
(\{n:s_{n}\neq 0\})=0,$ then $\int\limits_{B}s_{n}\left\Vert T_{\gamma ,%
\mathbf{a}}\right\Vert (dn)=0,$ for any $B\in \mathcal{B}(\mathbb{N}).$ More
importantly, for the other direction, if $\int\limits_{A}s_{n}\left\Vert
T_{\gamma ,\mathbf{a}}\right\Vert (dn)=0,$ $\forall A\in \mathcal{B}(\mathbb{%
N}),$ then $s_{n}=0$ a.e. $[\left\Vert T_{\gamma ,\mathbf{a}}\right\Vert ].$
In other words, if $I_{\gamma ,\mathbf{a}}^{\mathbf{p},A}=I_{\gamma ,\mathbf{%
a}}^{\mathbf{q},A},$ $\forall A\in \mathcal{B}(\mathbb{N}),$ then $\mathbf{p}%
=\mathbf{q}$ a.e. $[\left\Vert T_{\gamma ,\mathbf{a}}\right\Vert ].$

\item Equality a.e.: If $f_{n}=g_{n},$ a.e. $[T_{\gamma ,\mathbf{a}}],$ $%
\forall n\in \mathbb{N},$ and $I_{\gamma ,\mathbf{a}}^{\mathbf{f},B}$
exists, then $I_{\gamma ,\mathbf{a}}^{\mathbf{g},B}$ exists, and $I_{\gamma ,%
\mathbf{a}}^{\mathbf{f},B}=I_{\gamma ,\mathbf{a}}^{\mathbf{g},B}.$

\item Positivity: If $T_{\gamma ,\mathbf{a}}\geq 0$ and$\ s_{n}\geq 0,$ $%
\forall n\in \mathbb{N},$ and $T_{\gamma ,\mathbf{a}}(\{n:s_{n}>0\})>0,$
then $I_{\gamma ,\mathbf{a}}^{\mathbf{s},B}>0,$ provided the integral exists.

\item Finiteness: If $T_{\gamma ,\mathbf{a}}\geq 0,$ $s_{n}\geq 0,$ $\forall
n\in \mathbb{N},$ and $I_{\gamma ,\mathbf{a}}^{\mathbf{s},\mathbb{N}}<\infty
,$ then $s_{n}<\infty $ a.e. $[T_{\gamma ,\mathbf{a}}].$ To see this write%
\begin{eqnarray*}
T_{\gamma ,\mathbf{a}}(\{s_{n} &=&\infty \})=T_{\gamma ,\mathbf{a}}\left(
\bigcap\limits_{k\geq 1}\{s_{n}\geq k\}\right)
=\int\limits_{\bigcap\limits_{k\geq 1}\{s_{n}\geq k\}}1T_{\gamma ,\mathbf{a}%
}(dn) \\
&\leq &\int\limits_{\{s_{n}\geq k\}}1T_{\gamma ,\mathbf{a}}(dn)=\int I\left(
\frac{s_{n}}{k}\geq 1\right) T_{\gamma ,\mathbf{a}}(dn)\leq \frac{1}{k}\int
s_{n}T_{\gamma ,\mathbf{a}}(dn)=\frac{1}{k}I_{\gamma ,\mathbf{a}}^{\mathbf{s}%
,\mathbb{N}}\rightarrow 0,
\end{eqnarray*}%
as $k\rightarrow \infty .$

\item Integral support: Assume that $s_{n}\geq 0,$ $\forall n\in \mathbb{N},$
and $\left\Vert T_{\gamma ,\mathbf{a}}\right\Vert (A^{c})=0.$ Then $%
I_{\gamma ,\mathbf{a}}^{\mathbf{s},\mathbb{N}}=I_{\gamma ,\mathbf{a}}^{%
\mathbf{s},A},$ by definition, since $\left\Vert T_{\gamma ,\mathbf{a}%
}\right\Vert $ vanishes outside $A$.

\item Dominated Convergence Theorem (DCT): Assume that a sequence of
sequences $s_{n}^{(k)}:$ $\mathbb{N}\rightarrow \Re ,$ with $\mathbf{s}%
^{(k)}=(s_{0}^{(k)},s_{1}^{(k)},...),$ $\forall k\in \mathbb{N},$ converges
to a sequence $s_{n}$, i.e., $s_{n}^{(k)}\rightarrow s_{n},$ as $%
k\rightarrow \infty ,$ pointwise in $n\in \mathbb{N}$, and that there exists
a sequence $\mathbf{g}$ that is integrable $[T_{\gamma ,\mathbf{a}}],$ such
that $\left\vert s_{n}^{(k)}\right\vert \leq g_{n},$ $\forall k\in \mathbb{N}%
.$ Then $\mathbf{s}$ is integrable $[T_{\gamma ,\mathbf{a}}],$ and%
\begin{equation*}
\underset{k\rightarrow \infty }{\lim }\int\limits_{B}s_{n}^{(k)}T_{\gamma ,%
\mathbf{a}}(dn)=\int\limits_{B}s_{n}T_{\gamma ,\mathbf{a}}(dn).
\end{equation*}%
The proof is immediate by definition and an appeal to the usual DCT (twice)
for the integrals $P_{\gamma ,\mathbf{a}}^{\mathbf{s},B}$ and $N_{\gamma ,%
\mathbf{a}}^{\mathbf{s},B}$.

\item Dimension-wise multiplication: From Equation (\ref{Int1}), we have%
\begin{equation*}
I_{\gamma ,\mathbf{a}}^{\mathbf{s}\cdot \mathbf{f}\cdot \mathbf{g}%
,B}=\int\limits_{B}s_{n}f_{n}g_{n}T_{\gamma ,\mathbf{a}}(dn)=I_{\gamma ,%
\mathbf{g\cdot a}}^{\mathbf{s}\cdot \mathbf{f},B}=I_{\gamma ,\mathbf{f\cdot
g\cdot a}}^{\mathbf{s},B}=I_{\gamma ,\mathbf{s}\cdot \mathbf{f\cdot g\cdot a}%
}^{\mathbf{1},B},
\end{equation*}%
so that we can iterate the TI of a sequence and absorb it into the Taylor
measure.

\item Convolution: The convolution (Cauchy product) of $\mathbf{f}$ and $%
\mathbf{g}$ is given by the product of two TIs as%
\begin{eqnarray*}
I_{\gamma ,\mathbf{a}}^{\mathbf{f},\mathbb{N}}I_{\gamma ,\mathbf{a}}^{%
\mathbf{g},\mathbb{N}} &=&\sum\limits_{n\in \mathbb{N}}f_{n}a_{n}\frac{%
\gamma ^{n}}{n!}\sum\limits_{k\in \mathbb{N}}g_{k}a_{k}\frac{\gamma ^{k}}{k!}%
=\sum\limits_{n\in \mathbb{N}}\sum\limits_{k=0}^{n}\frac{%
f_{k}a_{k}g_{n-k}a_{n-k}}{k!(n-k)!}\gamma ^{n} \\
&=&\sum\limits_{n\in \mathbb{N}}\sum\limits_{k=0}^{n}\frac{%
f_{k}a_{k}g_{n-k}a_{n-k}}{a_{n}}\frac{n!}{k!(n-k)!}a_{n}\frac{\gamma ^{n}}{n!%
}=\sum\limits_{n\in \mathbb{N}}p_{n}a_{n}\frac{\gamma ^{n}}{n!},
\end{eqnarray*}%
where $p_{n}=\sum\limits_{k=0}^{n}\mathcal{C}_{k}^{n}\frac{%
f_{k}a_{k}g_{n-k}a_{n-k}}{a_{n}},$ with $\mathcal{C}_{k}^{n}=\frac{n!}{%
k!(n-k)!},$ the binomial coefficient, and therefore%
\begin{equation}
I_{\gamma ,\mathbf{a}}^{\mathbf{f},\mathbb{N}}I_{\gamma ,\mathbf{a}}^{%
\mathbf{g},\mathbb{N}}=I_{\gamma ,\mathbf{a}}^{\mathbf{p},\mathbb{N}},
\label{Convol1}
\end{equation}%
provided that $\gamma ,a_{n}\neq 0,$ $\forall n\in \mathbb{N}.$

\item Integral form and differentiation: For any $\gamma \in \Re ,$ and for
some $\gamma _{0}\in \Re ,$ it is easy to see that%
\begin{equation*}
\gamma ^{n}=\gamma _{0}^{n}+n\int\limits_{\gamma _{0}}^{\gamma }x^{n-1}\mu
_{1}(dx),
\end{equation*}%
and after an appeal to Dominated Convergence Theorem we can write%
\begin{eqnarray*}
I_{\gamma ,\mathbf{a}}^{\mathbf{s},\mathbb{N}} &=&\sum\limits_{n=0}^{+\infty
}s_{n}a_{n}\frac{\gamma ^{n}}{n!}=\sum\limits_{n=0}^{+\infty }\frac{%
s_{n}a_{n}}{n!}\left[ \gamma _{0}^{n}+n\int\limits_{\gamma _{0}}^{\gamma
}x^{n-1}\mu _{1}(dx)\right] \\
&=&I_{\gamma _{0},\mathbf{a}}^{\mathbf{s},\mathbb{N}}+\int\limits_{\gamma
_{0}}^{\gamma }\sum\limits_{n=1}^{+\infty }\frac{s_{n}a_{n}}{(n-1)!}%
x^{n-1}\mu _{1}(dx)\overset{k=n-1}{=}I_{\gamma _{0},\mathbf{a}}^{\mathbf{s},%
\mathbb{N}}+\int\limits_{\gamma _{0}}^{\gamma }\sum\limits_{k=0}^{+\infty }%
\frac{s_{k+1}a_{k+1}}{k!}x^{k}\mu _{1}(dx).
\end{eqnarray*}%
Consequently, we have%
\begin{equation}
I_{\gamma ,\mathbf{a}}^{\mathbf{s},\mathbb{N}}=I_{\gamma _{0},\mathbf{a}}^{%
\mathbf{s},\mathbb{N}}+\int\limits_{\gamma _{0}}^{\gamma }I_{x,\tau (\mathbf{%
s\cdot a})}^{\mathbf{1},\mathbb{N}}\mu _{1}(dx),  \label{Fundamental1}
\end{equation}%
where $\tau (\mathbf{s})=(s_{1},s_{2},...),$ the shift operator of the
sequence $\mathbf{s}.$ In particular, for $\gamma _{0}=0$, we obtain%
\begin{equation}
I_{\gamma ,\mathbf{a}}^{\mathbf{s},\mathbb{N}}=s_{0}a_{0}+\int\limits_{0}^{%
\gamma }I_{x,\tau (\mathbf{s\cdot a})}^{\mathbf{1},\mathbb{N}}\mu _{1}(dx),
\label{IntegralForm1}
\end{equation}%
for all $\gamma \in \Re ,$ since $I_{0,\mathbf{a}}^{\mathbf{s},\mathbb{N}%
}=s_{0}a_{0}$ (with the convention $0^{0}=1)$, and as a result, we can write%
\begin{equation}
\frac{dI_{x,\mathbf{a}}^{\mathbf{s},\mathbb{N}}}{dx}=I_{x,\tau (\mathbf{%
s\cdot a})}^{\mathbf{1},\mathbb{N}},  \label{DerivForm1}
\end{equation}%
for all $x\in \Re .$

\item Generating functions: Take $B=\mathbb{N},$ and assume $x\in \mathbb{M}$%
. The TI in this case is related to the ordinary generating function (OGF)
via%
\begin{equation}
G(x)=\sum\limits_{n\in \mathbb{N}}s_{n}x^{n}=\sum\limits_{n\in \mathbb{N}%
}s_{n}n!\frac{x^{n}}{n!}=I_{x,\mathbf{n!}}^{\mathbf{s},\mathbb{N}},
\end{equation}%
where $\mathbf{n!}=[0!,1!,2!,...],$ and similarly, the exponential
generating function (EGF) can be written as
\begin{equation}
E(x)=\sum\limits_{n\in \mathbb{N}}s_{n}\frac{x^{n}}{n!}=I_{x,\mathbf{1}}^{%
\mathbf{s},\mathbb{N}},
\end{equation}%
and therefore, the TI unifies these two important concepts. Then we can
immediately take advantage of all the theory in the literature on OGFs and
EGFs. For example, we can connect TIs using the relation of $G(x)$ and $E(x)
$\ given by%
\begin{equation*}
I_{x,\mathbf{n!}}^{\mathbf{s},\mathbb{N}}=G(x)=\int\limits_{0}^{+\infty
}e^{-t}E(tx)\mu _{1}(dt)=\int\limits_{0}^{+\infty }e^{-t}I_{tx,\mathbf{1}}^{%
\mathbf{s},\mathbb{N}}\mu _{1}(dt).
\end{equation*}
\end{enumerate}
\end{remark}

We discuss a general result involving TIs and real sequences in order to
appreciate first applications of the new integral and its usefulness. The
following provides a connection between absolute convergence of a real
series and the TI.

\begin{theorem}[Characterization of Absolutely Convergent Series]
Consider a sequence $s_{n}:\mathbb{N}\rightarrow \Re ,$ with $\mathbf{s}%
=[s_{0},$ $s_{1},$ $...].$ If the series $\sum\limits_{n=0}^{+\infty }s_{n}$
converges absolutely then there exists a Taylor measure $T_{\gamma ,\mathbf{a%
}}\in \mathcal{T}^{\mathcal{F}}$, $\gamma \neq 0,$ with $T_{\left\vert
\gamma \right\vert ,\left\vert \mathbf{a}\right\vert }(\mathbb{N})=+\infty ,$
(i.e., $\sum\limits_{n=0}^{+\infty }a_{n}\frac{\gamma ^{n}}{n!}$ is
conditionally convergent), such that the Taylor integral $I_{\left\vert
\gamma \right\vert ,\left\vert \mathbf{a}\right\vert }^{\left\vert \mathbf{s}%
\right\vert ,\mathbb{N}}$ of $\left\vert \mathbf{s}\right\vert $\ with
respect to $T_{\left\vert \gamma \right\vert ,\left\vert \mathbf{a}%
\right\vert },$ is finite.\newline
Partial Converse: The series $\sum\limits_{n=0}^{+\infty }s_{n}$ converges
absolutely if there exists a Taylor measure $T_{\gamma ,\mathbf{a}}\in
\mathcal{T}^{\mathcal{F}}$, $\gamma \neq 0,$ with $T_{\left\vert \gamma
\right\vert ,\left\vert \mathbf{a}\right\vert }(\mathbb{N})=+\infty ,$ $%
I_{\left\vert \gamma \right\vert ,\left\vert \mathbf{a}\right\vert
}^{\left\vert \mathbf{s}\right\vert ,\mathbb{N}}<+\infty ,$ and%
\begin{equation}
\limsup_{n\rightarrow +\infty }\left\vert \frac{(n+1)a_{n}}{\gamma a_{n+1}}%
\right\vert \leq 1.  \label{conda_n}
\end{equation}%
and in addition, one of the following growth conditions holds:%
\begin{equation}
(A):\frac{|s_{n+1}||a_{n+1}|}{|s_{n}||a_{n}|}=o(n),  \label{condA}
\end{equation}%
or%
\begin{equation}
(B):\frac{|s_{n+1}|}{|s_{n}|}=O(n^{\alpha }),\text{ and }\frac{|a_{n+1}|}{%
|a_{n}|}=O(n^{\beta }),  \label{condB}
\end{equation}%
with $\alpha ,\beta \geq 0,$ $\alpha +\beta <1$ (polynomial bounds on the
successive ratios)$,$\newline
or%
\begin{equation}
(C):s_{n}\sim C_{s}n^{p},\text{ and }a_{n}\sim C_{a}n^{q},  \label{condC}
\end{equation}%
for some constants $C_{s},C_{a}>0,$ and integers $p,q\geq 1$ (polynomial
growth),\newline
or%
\begin{equation}
(D):s_{n}\sim C_{s}R_{s}^{n}n^{p},\text{ and }a_{n}\sim C_{a}R_{a}^{n}n^{q},
\label{condD}
\end{equation}%
for some constants $R_{s},C_{s},R_{a},C_{a}>0,$ and integers $p,q\geq 1$
(exponential growth).
\end{theorem}

\begin{proof}
($\Rightarrow $) Assume that $\sum\limits_{n=0}^{+\infty }|s_{n}|<+\infty ,$
and take $\mathbf{a}=[a_{0},$ $a_{1},...],$ with $a_{0}=0,$ and $a_{n}=\frac{%
n!}{\gamma ^{n}}\frac{(-1)^{n}}{\sqrt{n}},$ $\forall n\geq 1,$ and choose
any $\gamma \neq 0.$ First we note that%
\begin{equation*}
T_{\gamma ,\mathbf{a}}(\mathbb{N})=\sum\limits_{n=0}^{+\infty }a_{n}\frac{%
\gamma ^{n}}{n!}=\sum\limits_{n=1}^{+\infty }\frac{n!}{\gamma ^{n}}\frac{%
(-1)^{n}}{\sqrt{n}}\frac{\gamma ^{n}}{n!}=\sum\limits_{n=1}^{+\infty }\frac{%
(-1)^{n}}{\sqrt{n}}<+\infty ,
\end{equation*}%
and therefore, $T_{\gamma ,\mathbf{a}}\in \mathcal{T}^{\mathcal{F}}.$
However,%
\begin{equation*}
T_{\left\vert \gamma \right\vert ,\left\vert \mathbf{a}\right\vert }(\mathbb{%
N})=\sum\limits_{n=0}^{+\infty }\left\vert a_{n}\right\vert \frac{\left\vert
\gamma \right\vert ^{n}}{n!}<\sum\limits_{n=1}^{+\infty }\frac{n!}{%
\left\vert \gamma \right\vert ^{n}}\frac{\left\vert (-1)^{n}\right\vert }{%
\sqrt{n}}\frac{\left\vert \gamma \right\vert ^{n}}{n!}=\sum\limits_{n=1}^{+%
\infty }\frac{1}{\sqrt{n}}=+\infty ,
\end{equation*}%
and consequently, we can write%
\begin{eqnarray*}
I_{\left\vert \gamma \right\vert ,\left\vert \mathbf{a}\right\vert
}^{\left\vert \mathbf{s}\right\vert ,\mathbb{N}}
&=&\sum\limits_{n=0}^{+\infty }\left\vert s_{n}\right\vert \left\vert
a_{n}\right\vert \frac{\left\vert \gamma \right\vert ^{n}}{n!}%
=\sum\limits_{n=1}^{+\infty }\left\vert s_{n}\right\vert \frac{n!}{%
\left\vert \gamma \right\vert ^{n}}\frac{\left\vert (-1)^{n}\right\vert }{%
\sqrt{n}}\frac{\gamma ^{n}}{n!} \\
&=&\sum\limits_{n=1}^{+\infty }\left\vert s_{n}\right\vert \frac{1}{\sqrt{n}}%
<\sum\limits_{n=1}^{+\infty }\left\vert s_{n}\right\vert <+\infty ,
\end{eqnarray*}%
so that we can choose this conditionally convergent $T_{\gamma ,\mathbf{a}%
}\in \mathcal{T}^{\mathcal{F}}$ to obtain $I_{\left\vert \gamma \right\vert
,\left\vert \mathbf{a}\right\vert }^{\left\vert \mathbf{s}\right\vert ,%
\mathbb{N}}<\infty .$\newline
($\Leftarrow $) Define%
\begin{equation*}
b_{n}:=|s_{n}|\,|a_{n}|,
\end{equation*}%
and assume that there exists a Taylor measure $T_{\left\vert \gamma
\right\vert ,\left\vert \mathbf{a}\right\vert }\in \mathcal{T}^{\mathcal{F}}$%
,\ with $T_{\left\vert \gamma \right\vert ,\left\vert \mathbf{a}\right\vert
}(\mathbb{N})=+\infty ,$ such that the Taylor integral
\begin{equation*}
I_{\left\vert \gamma \right\vert ,\left\vert \mathbf{a}\right\vert
}^{\left\vert \mathbf{s}\right\vert ,\mathbb{N}}=\sum\limits_{n=0}^{+\infty
}\left\vert s_{n}\right\vert \left\vert a_{n}\right\vert \frac{\left\vert
\gamma \right\vert ^{n}}{n!}<+\infty ,
\end{equation*}%
is finite and one of the growth conditions $(A)-(D)$ is satisfied$.$\newline
We will show that $\sum\limits_{n=0}^{+\infty }|s_{n}|<+\infty $, by
employing the usual ratio test, i.e., we need to prove that%
\begin{equation*}
\underset{n\rightarrow +\infty }{\lim }\left\vert \frac{s_{n+1}}{s_{n}}%
\right\vert <1.
\end{equation*}%
Now if%
\begin{equation*}
\underset{n\rightarrow +\infty }{\lim }\left\vert \frac{s_{n+1}a_{n+1}\frac{%
\gamma ^{n+1}}{(n+1)!}}{s_{n}a_{n}\frac{\gamma ^{n}}{n!}}\right\vert
<1\Rightarrow \underset{n\rightarrow +\infty }{\lim }\left\vert \frac{s_{n+1}%
}{s_{n}}\right\vert <\underset{n\rightarrow +\infty }{\lim }\left\vert \frac{%
(n+1)a_{n}}{\gamma a_{n+1}}\right\vert .
\end{equation*}%
and using condition (\ref{conda_n}) we have%
\begin{equation*}
\underset{n\rightarrow +\infty }{\lim }\left\vert \frac{(n+1)a_{n}}{\gamma
a_{n+1}}\right\vert \leq \limsup_{n\rightarrow +\infty }\left\vert \frac{%
(n+1)a_{n}}{\gamma a_{n+1}}\right\vert \leq 1,
\end{equation*}%
such that%
\begin{equation*}
\underset{n\rightarrow +\infty }{\lim }\left\vert \frac{s_{n+1}}{s_{n}}%
\right\vert <1,
\end{equation*}%
as required.\newline
Therefore, we need to show that%
\begin{equation*}
\limsup_{n\rightarrow +\infty }\left\vert \frac{b_{n+1}\frac{\gamma ^{n+1}}{%
(n+1)!}}{b_{n}\frac{\gamma ^{n}}{n!}}\right\vert =\limsup_{n\rightarrow
+\infty }\left\vert \frac{b_{n+1}\gamma }{(n+1)b_{n}}\right\vert <1,
\end{equation*}%
or equivalently that%
\begin{equation}
\limsup_{n\rightarrow +\infty }\left\vert \frac{b_{n+1}}{(n+1)b_{n}}%
\right\vert <\frac{1}{\left\vert \gamma \right\vert },  \label{condineq}
\end{equation}%
for fixed $\gamma \neq 0.$\newline
\textbf{Under Condition (A)}: From the condition, dividing by $n+1$ gives%
\begin{equation*}
\frac{|b_{n+1}|}{(n+1)|b_{n}|}=o(1),
\end{equation*}%
and therefore
\begin{equation*}
\limsup_{n\rightarrow +\infty }\frac{|b_{n+1}|}{(n+1)|b_{n}|}=0.
\end{equation*}%
Since $\gamma $ is fixed and $\gamma \neq 0$,
\begin{equation*}
0<\frac{1}{|\gamma |},
\end{equation*}%
and thus the desired inequality (\ref{condineq}) follows.\newline
\textbf{Under Condition (B)}: From the condition, for sufficiently large $n$%
, there exists some constants $C_{s},C_{a}>0,$ such that
\begin{equation*}
\frac{|s_{n+1}|}{|s_{n}|}\leq C_{s}(n+1)^{\alpha },\text{ and }\frac{%
|a_{n+1}|}{|a_{n}|}\leq C_{a}(n+1)^{\beta },
\end{equation*}%
and putting these bounds together gives%
\begin{equation*}
\frac{|s_{n+1}||a_{n+1}|}{(n+1)|s_{n}||a_{n}|}\leq C_{s}C_{a}(n+1)^{\alpha
+\beta -1},
\end{equation*}%
and under the condition $\alpha +\beta <1,$ we obtain%
\begin{equation*}
(n+1)^{\alpha +\beta -1}\longrightarrow 0,
\end{equation*}%
as $n\rightarrow +\infty .$ Consequently,
\begin{equation*}
\limsup_{n\rightarrow +\infty }\frac{|s_{n+1}||a_{n+1}|}{(n+1)|s_{n}||a_{n}|}%
=0,
\end{equation*}%
and equation (\ref{condineq}) is satisfied. This condition is substantially
more general than requiring both successive ratios to be bounded, e.g., the
special case
\begin{equation*}
\frac{|s_{n+1}|}{|s_{n}|}=O(1),\text{ and }\frac{|a_{n+1}|}{|a_{n}|}=O(1),
\end{equation*}%
such that%
\begin{equation*}
\frac{|s_{n+1}||a_{n+1}|}{(n+1)|s_{n}||a_{n}|}=O\left( \frac{1}{n+1}\right) .
\end{equation*}%
\newline
\textbf{Under Condition (C)}: Under these asymptotic estimates, we have
\begin{equation*}
\frac{|s_{n+1}|}{|s_{n}|}\longrightarrow 1,\text{ and }\frac{|a_{n+1}|}{%
|a_{n}|}\longrightarrow 1,
\end{equation*}%
and therefore the asymptotic estimate of the ratio in question is%
\begin{equation*}
\frac{|s_{n+1}||a_{n+1}|}{(n+1)|s_{n}||a_{n}|}\sim \frac{1}{n+1}.
\end{equation*}%
Consequently
\begin{equation*}
\underset{n\rightarrow +\infty }{\lim }\left\vert \frac{b_{n+1}}{(n+1)b_{n}}%
\right\vert =0,
\end{equation*}%
and thus equation (\ref{condineq}) is established. Note here that polynomial
growth is easily dominated by the factorial in the denominator of the
sequence
\begin{equation*}
u_{n}=s_{n}a_{n}\frac{\gamma ^{n}}{n!}.
\end{equation*}%
\newline
\textbf{Under Condition (D)}: Using the condition we can write
\begin{equation*}
\frac{|s_{n+1}|}{|s_{n}|}\sim R_{s}\left( 1+\frac{1}{n}\right) ^{p},
\end{equation*}%
and
\begin{equation*}
\frac{|a_{n+1}|}{|a_{n}|}\sim R_{a}\left( 1+\frac{1}{n}\right) ^{q},
\end{equation*}%
such that the ratio has asymptotic estimate
\begin{equation*}
\frac{|s_{n+1}||a_{n+1}|}{(n+1)|s_{n}||a_{n}|}\sim \frac{R_{s}R_{a}}{n+1},
\end{equation*}%
and once again
\begin{equation*}
\underset{n\rightarrow +\infty }{\lim }\left\vert \frac{b_{n+1}}{(n+1)b_{n}}%
\right\vert =0,
\end{equation*}%
and thus equation (\ref{condineq}) is satisfied.
\end{proof}

\subsection{Examples and connection with Fourier transform}

The TI\ provides a unifying framework that allows us to study many important
series from real and complex analysis, since they are special cases, as we
show next.

\begin{example}[Special Cases]
\label{AppTMIs}We collect some important series as special cases of the TI.
In what follows, take any sequences $s_{n},f_{n}:$ $\mathbb{N}\rightarrow
\mathbb{M},$ with $\mathbf{s}=(s_{0},s_{1},...),$ $\mathbf{f}%
=(f_{0},f_{1},...),$ and a Taylor measure $T_{\gamma ,\mathbf{a}}(B),$ for
all $B\in \mathcal{B}(\mathbb{N}),$ such that the TI $I_{\gamma ,\mathbf{a}%
}^{\mathbf{s},B}=\sum\limits_{n\in B}s_{n}a_{n}\frac{\gamma ^{n}}{n!},$
exists and is finite.

\begin{enumerate}
\item Riemann's $\zeta $ function: Take $B=\mathbb{N},$ $\gamma =1,$\ and
define the sequence $s_{0}(z)=0,$ $s_{n}(z)=n!/n^{z},$ $n\in \mathbb{N}%
^{+}=\{1,2,...\},$ for fixed $z\in \mathbb{C},$ so that equation (\ref%
{Int1signedmeas}) yields Riemann's $\zeta $ function (see \cite%
{borwein2000efficient}, \cite{guillera2008double}) defined for any complex
number $z=\sigma +it\in \mathbb{C}$ by%
\begin{equation}
\zeta (z)=\sum\limits_{n=1}^{+\infty }\frac{1}{n^{z}}=\sum\limits_{n=1}^{+%
\infty }\frac{1}{n^{\sigma }}n^{-it}=I_{1,\mathbf{1}}^{\mathbf{s}(z),\mathbb{%
N}},  \label{zetafunc}
\end{equation}%
where the sum converges if $\sigma >1.$ For a recent treatment of the $\zeta
$ function using series expansions see \cite{young2023global}.

\item Dirichlet's $\eta $ function: Recall Dirichlet's $\eta $ function,
also known as the alternating $\zeta $ function (see \cite{sondow2003zeros},
\cite{guillera2008double}). Take $B=\mathbb{N},$ $\gamma =-1,$\ and define
the sequence $s_{0}(z)=0,$ $s_{n}(z)=n!/n^{z},$ $n\in \mathbb{N}^{+},$ for
fixed $z=\sigma +it\in \mathbb{C},$ as in the previous example, so that $%
\eta $ is written as the TI given by%
\begin{equation*}
\eta (z)=\sum\limits_{n=1}^{+\infty }\frac{(-1)^{n+1}}{n^{z}}%
=-\sum\limits_{n=1}^{+\infty }\frac{(-1)^{n}}{n^{z}}=-I_{-1,\mathbf{1}}^{%
\mathbf{s}(z),\mathbb{N}},
\end{equation*}%
which converges for $\sigma >0.$

\item Dirichlet series: Take $B=\mathbb{N},$ $\gamma =1,$ define the
sequence $s_{0}(z)=0,$ $s_{n}(z)=n!/n^{z},$ $n\in \mathbb{N}^{+},$ for fixed
$z=\sigma +it\in \mathbb{C},$ and recall the Dirichlet series is given by%
\begin{equation*}
\mathcal{D}^{\mathbf{s}}(z)=\sum\limits_{n=1}^{+\infty }\frac{s_{n}}{n^{z}}%
=\sum\limits_{n\in B}\frac{n!}{n^{z}}s_{n}\frac{1^{n}}{n!}=I_{1,\mathbf{s}}^{%
\mathbf{s}(z),\mathbb{N}},
\end{equation*}%
which is once again a TI, and converges absolutely when $s_{n}$ are
uniformly bounded, for $\sigma >1.$ When $s_{n}=O(n^{k}),$ $\mathcal{D}^{%
\mathbf{s}}(z)$ converges absolutely in the half plane $\sigma >k+1.$
Clearly, Riemann's $\zeta $ and Dirichlet's $\eta $ functions are special
cases of this TI (see \cite{freitag2005complex}, Chapter VII, for a general
treatment of this TI).

\item Hypergeometric series: Set $B=\mathbb{N},$ $f_{n}=n!s_{n},$ and
further assume that%
\begin{equation*}
\frac{s_{n+1}}{s_{n}}=\frac{A(n)}{B(n)},
\end{equation*}%
with $A(n)$ and $B(n)$ polynomials in $n$, for all $n\in \mathbb{N}$. Then
the TI becomes the hypergeometric series%
\begin{equation*}
\sum\limits_{n\in B}s_{n}\gamma ^{n}=\sum\limits_{n\in \mathbb{N}}f_{n}\frac{%
\gamma ^{n}}{n!}=I_{\gamma ,\mathbf{1}}^{\mathbf{f},\mathbb{N}},
\end{equation*}%
which can be extended to the generalized hypergeometric function; recall the
Pochhammer symbols defined by%
\begin{eqnarray*}
(a)_{0} &=&1, \\
(a)_{n} &=&a(a+1)(a+2)...(a+n-1),\text{ }n\in \mathbb{N}^{+},
\end{eqnarray*}%
$a\in \Re ,$ and write%
\begin{equation*}
I_{\gamma ,\mathbf{1}}^{\mathbf{f},\mathbb{N}}=\text{ }%
_{p}F_{q}(a_{1},...,a_{p};b_{1},...,b_{q};\gamma )=\sum\limits_{n\in \mathbb{%
N}}s_{n}\frac{\gamma ^{n}}{n!},
\end{equation*}%
with%
\begin{equation*}
s_{n}=\frac{(a_{1})_{n}...(a_{p})_{n}}{(b_{1})_{n}...(b_{q})_{n}},
\end{equation*}%
for all $n\in \mathbb{N}$, where
\begin{equation*}
b_{j}\notin \{0,-1,-2,\ldots \}.
\end{equation*}%
Note that the convergence domain also depends on $p$ and $q$, as follows: if
$p<q+1$, the series is entire in $z$; if $p=q+1$, the radius of convergence
is $1$; if $p>q+1$, the series generally diverges except in special cases.

\item Discrete Fourier transform: Assume that $B_{N}=\{0,1,...,N-1\},$ for
some $N\in \mathbb{N},$ let $c_{k}=e^{-i2\pi \frac{k}{N}},$ $k\in \mathbb{N},
$ and consider the transformation of the sequence $\mathbf{s}$ defined by%
\begin{equation}
f_{k}=\sum\limits_{n=0}^{N-1}s_{n}e^{-i2\pi \frac{k}{N}n}=\sum\limits_{n\in
B_{N}}s_{n}n!\frac{\left( e^{-i2\pi \frac{k}{N}}\right) ^{n}}{n!}%
=\sum\limits_{n\in B_{N}}s_{n}a_{n}\frac{\left( c_{k}\right) ^{n}}{n!}%
=I_{c_{k},\mathbf{n!}}^{\mathbf{s},B_{N}},  \label{Fourier1}
\end{equation}%
where $a_{n}=n!,$ $n\in \mathbb{N},$ which is $N$-periodic in the index $k$
(i.e., $f_{k+N}=f_{k}$). The sequence $\{f_{k}\}$ is the celebrated discrete
Fourier transform (DFT) which is a special case of the TI, and can be
inverted to give the sequence $\mathbf{s}$ based on the sequence $\mathbf{f}$
via the inverse DFT (IDFT) given by%
\begin{equation}
s_{n}=\frac{1}{N}\sum\limits_{k=0}^{N-1}f_{k}e^{i2\pi \frac{k}{N}n}=\frac{1}{%
N}\sum\limits_{k\in B_{N}}f_{k}k!\frac{\left( e^{i2\pi \frac{n}{N}}\right)
^{k}}{k!}=\frac{1}{N}\sum\limits_{k\in B_{N}}f_{k}a_{k}\frac{\left(
d_{n}\right) ^{k}}{k!}=\frac{1}{N}I_{d_{n},\mathbf{k!}}^{\mathbf{f},B_{N}},
\label{InvFourier1}
\end{equation}%
where $d_{n}=c_{n}^{-1}=\overline{c}_{n}=e^{i2\pi \frac{n}{N}},$ $n\in
\mathbb{N}.$ For some methodologies and applications of the DFT see \cite%
{steidl1991polynomial}, \cite{crandall1994discrete}, \cite{bansal2007finding}%
, \cite{logofuatu2011fourier}, \cite{d2017shuffling}, \cite%
{jiang2023efficient}, \cite{yang2024sparse}, \cite{yang2024local}, \cite%
{alrzqi2024efficient}, \cite{grohs2025completeness} and the references
therein. Clearly, the IDFT is also $N$-periodic in the index $k$, and once
again a special case of the TI. Note here that the integrating measure for
the DFT is the Taylor measure $T_{c_{k},\mathbf{n!}}(B_{N}),$ and for the
IDFT the Taylor measure $T_{\overline{c}_{n},\mathbf{k!}}(B_{N}).$ It has
been established in the literature that the DFT is unique (i.e., the only TI
with this property), in the sense that it is the only linear and invertible
map, modulo a permutation of indices, which carries convolution into the
componentwise product (see \cite{baraquin2023uniqueness} and the references
therein).

\item Euler's totient function: Consider the sequence $\phi (n):\mathbb{N}%
^{+}\rightarrow \mathbb{N}^{+},$ where $\phi (n)$ counts the number of
positive integers up to a given integer $n$ that are relatively prime to $n$%
, i.e., $\phi (n)$ is the number of integers $\#(k:k\leq n)$, for which the
greatest common divisor $gcd(n,k)$ is equal to 1. The function $\phi (n)$ is
known as Euler's totient function, and can be expressed in terms of prime
numbers as follows; let $n=p_{1}^{k_{1}}p_{2}^{k_{2}}...p_{r}^{k_{r}},$ the
prime factorization of $n,$ based on distinct prime numbers $%
p_{1},p_{2},...,p_{r},$ for some integers $k_{1},k_{2},...,k_{r}\in \mathbb{N%
}^{+},$ such that%
\begin{equation*}
\phi (n)=n\prod\limits_{\text{all }p|n}\left( 1-\frac{1}{p}\right)
=p_{1}^{k_{1}-1}(p_{1}-1)p_{2}^{k_{2}-1}(p_{2}-1)...p_{r}^{k_{r}-1}(p_{r}-1),
\end{equation*}%
where the product is taken over all primes $p$ that divide $n$, denoted by $%
p|n.$\newline
Letting $s_{k}=gcd(n,k),$ $k\in \{1,2,...,n\},$ we can write the latter in
terms of the DFT (see \cite{schramm2008fourier}) using%
\begin{equation*}
\phi (n)=\sum\limits_{k=1}^{n}gcd(k,n)e^{-i2\pi \frac{k}{n}%
}=\sum\limits_{k=1}^{n}s_{k}\left( e^{-i2\pi \frac{1}{n}}\right)
^{k}=I_{c_{1},\mathbf{k!}}^{\mathbf{s},B_{n}},
\end{equation*}%
where $B_{n}=\{1,...,n\},$ $c_{1}=e^{-\frac{2\pi }{n}i},$ and as a result,
Euler's totient function is a TI. Importantly, the TI corresponding to the
Dirichlet series for $\phi (n),$ can be expressed in terms of the TI of the $%
\zeta $ function via%
\begin{equation*}
I_{1,\mathbf{\phi }}^{\mathbf{s}(z),\mathbb{N}}=\sum\limits_{n=1}^{+\infty }%
\frac{\phi (n)}{n^{z}}=\frac{\zeta (z-1)}{\zeta (z)}=\frac{I_{1,\mathbf{1}}^{%
\mathbf{s}(z-1),\mathbb{N}}}{I_{1,\mathbf{1}}^{\mathbf{s}(z),\mathbb{N}}},
\end{equation*}%
where $\mathbf{\phi }=[0,\phi (1),\phi (2),...],$ $s_{0}(z)=0,$ $%
s_{n}(z)=n!/n^{z},$ $n\in \mathbb{N}^{+},$ for fixed $z=\sigma +it\in
\mathbb{C},$ with $\sigma >2.$
\end{enumerate}
\end{example}

Owing to the latter example on the DFT and IDFT, it is clear that there
exists a specific Taylor measure $T_{c_{k},\mathbf{n!}}(B_{N}),$ $%
c_{k}=e^{-i2\pi \frac{k}{N}},$ $k\in \mathbb{N},$ and its conjugate version $%
T_{\overline{c}_{k},\mathbf{n!}}(B_{N}),$ such that the TI of the sequence $%
\mathbf{s}$ given by $f_{k}=I_{c_{k},\mathbf{n!}}^{\mathbf{s},B_{N}}$ can be
inverted using the TI $s_{k}=I_{\overline{c}_{k},\mathbf{n!}}^{\mathbf{f}%
,B_{N}}$. The DFT has other desirable properties, such as being able to
uniquely identify a sequence, convert convolution to componentwise product,
and many more (e.g., as a TI it satisfies all properties of Remark \ref%
{PropsTMI}). Next, we propose and investigate another such important special
case of Taylor measure in $\mathcal{T}^{\mathcal{F}}.$

\section{The Discrete Taylor Transformation}

In order to motivate the creation of a special TI, we address uniqueness of
sequences via their TI values, in a similar fashion as with the Laplace or
Fourier transforms.

\begin{theorem}[Uniqueness via Taylor Integrals]
\label{UniqDTT}Consider two sequences $p_{n},q_{n}:\mathbb{N}\rightarrow \Re
,$ with $\mathbf{p}=[p_{0},$ $p_{1},$ $...],$ $\mathbf{q}=[q_{0},$ $q_{1},$ $%
...],$ and assume that $I_{\gamma ,\mathbf{a}(t)}^{\mathbf{p},\mathbb{N}%
}=I_{\gamma ,\mathbf{a}(t)}^{\mathbf{q},\mathbb{N}}<+\infty ,$ for all $t>0,$
for some $\gamma \neq 0,$ and $\mathbf{a}=[a_{0},$ $a_{1},$ $...],$ where $%
a_{n}(t)=n!/t^{n+1},$ $n\in \mathbb{N},$ and assume that $t>\left\vert
\gamma \right\vert .$ Then the two sequences coincide, i.e., $p_{n}=q_{n},$
for all $n\in \mathbb{N}.$
\end{theorem}

\begin{proof}
First note that
\begin{equation*}
\int\limits_{0}^{+\infty }x^{n}e^{-tx}\mu _{1}(dx)\overset{y=tx}{=}\frac{1}{%
t^{n+1}}\int\limits_{0}^{+\infty }y^{n}e^{-y}\mu _{1}(dy)=\frac{\Gamma (n+1)%
}{t^{n+1}}=\frac{n!}{t^{n+1}}=a_{n}(t),
\end{equation*}%
for all $t>0$. Since $I_{\gamma ,\mathbf{a}(t)}^{\mathbf{p},\mathbb{N}%
}=I_{\gamma ,\mathbf{a}(t)}^{\mathbf{q},\mathbb{N}}<+\infty ,$ then one of
the three conditions (a)-(c) of Definition \ref{TaylorMeasureIntegral}
holds, and therefore%
\begin{equation*}
I_{\gamma ,\mathbf{a}(t)}^{\mathbf{p},\mathbb{N}}=\frac{1}{t}\sum_{n\in
\mathbb{N}}p_{n}\left( \frac{\gamma }{t}\right) ^{n}=\frac{1}{t}\sum_{n\in
\mathbb{N}}p_{n}\left( \frac{\gamma }{t}\right) ^{n}<+\infty ,
\end{equation*}%
and%
\begin{equation*}
I_{\gamma ,\mathbf{a}(t)}^{\mathbf{q},\mathbb{N}}=\frac{1}{t}\sum_{n\in
\mathbb{N}}q_{n}\left( \frac{\gamma }{t}\right) ^{n}<+\infty ,
\end{equation*}%
for all $t>0.$ In addition, we notice that the integrator of both integrals $%
I_{\gamma ,\mathbf{a}(t)}^{\mathbf{p},\mathbb{N}}$ and $I_{\gamma ,\mathbf{a}%
(t)}^{\mathbf{q},\mathbb{N}}$\ is the Taylor measure%
\begin{equation*}
T_{\gamma ,\mathbf{a}(t)}(\mathbb{N})=\sum_{n\in \mathbb{N}}\frac{n!}{t^{n+1}%
}\frac{\gamma ^{n}}{n!}=\frac{1}{t}\sum_{n\in \mathbb{N}}(\gamma
/t)^{n}<+\infty ,
\end{equation*}%
which is well defined since $t>\left\vert \gamma \right\vert .$\newline
Recall Dominated Convergence Theorem (DCT). More precisely, for any sequence
of functions $f_{n}(x),$ assume that there exists a sequence of functions $%
h_{n}(x)$ such that
\begin{equation*}
\left\vert f_{n}(x)\right\vert \leq h_{n}(x),\forall n\in \mathbb{N}\text{
and almost all }x\in \Re ,
\end{equation*}%
and in addition, assume that%
\begin{equation}
\sum\limits_{n\in \mathbb{N}}h_{n}(x)<+\infty ,  \label{DCTCond2}
\end{equation}%
and%
\begin{equation}
\int\limits_{\Re }h_{n}(x)\mu _{1}(dx)<+\infty .  \label{DCTCond3}
\end{equation}%
Then we can swap the order of summation and integration and write%
\begin{equation*}
\sum\limits_{n\in \mathbb{N}}\int\limits_{\Re }f_{n}(x)\mu
_{p}(dx)=\int\limits_{\Re }\sum\limits_{n\in \mathbb{N}}f_{n}(x)\mu _{p}(dx).
\end{equation*}%
Let%
\begin{equation*}
f_{n}(x)=p_{n}\frac{\gamma ^{n}}{n!}x^{n}e^{-tx},
\end{equation*}%
for $x>0$.\newline
\textbf{Under condition (a):} Assume that%
\begin{equation*}
\left\vert p_{n}a_{n}(t)\right\vert \leq M\Rightarrow \left\vert
p_{n}n!/t^{n+1}\right\vert \leq M\Rightarrow \left\vert p_{n}\right\vert
\leq \frac{Mt^{n+1}}{n!},
\end{equation*}%
$\forall n\in \mathbb{N}$, and for some $M>0.$ Consequently,%
\begin{equation*}
\left\vert f_{n}(x)\right\vert \leq \left\vert p_{n}\frac{\gamma ^{n}}{n!}%
x^{n}e^{-tx}\right\vert \leq \frac{\left\vert p_{n}\right\vert \left\vert
\gamma \right\vert ^{n}}{n!}x^{n}e^{-tx}\leq \frac{\left\vert
p_{n}\right\vert t^{n}}{n!}x^{n}e^{-tx}\leq \frac{Mt^{2n+1}}{\left(
n!\right) ^{2}}x^{n}e^{-tx}=h_{n}(x),
\end{equation*}%
where $t>\left\vert \gamma \right\vert ,$ it follows that%
\begin{equation*}
\int\limits_{\Re }h_{n}(x)\mu _{1}(dx)=\int\limits_{0}^{+\infty }\frac{%
Mt^{2n+1}}{\left( n!\right) ^{2}}x^{n}e^{-tx}\mu _{1}(dx)=\frac{Mt^{2n+1}}{%
\left( n!\right) ^{2}}\int\limits_{0}^{+\infty }x^{n}e^{-tx}\mu
_{1}(dx)<+\infty ,
\end{equation*}%
and%
\begin{eqnarray*}
\sum\limits_{n\in \mathbb{N}}h_{n}(x) &=&\sum\limits_{n\in \mathbb{N}}\frac{%
Mt^{2n+1}}{\left( n!\right) ^{2}}x^{n}e^{-tx}=Mte^{-tx}\sum\limits_{n\in
\mathbb{N}}\frac{\left( t^{2}x\right) ^{n}}{\left( n!\right) ^{2}}\leq
Mte^{-tx}\sum\limits_{n\in \mathbb{N}}\frac{\left( t^{2}x\right) ^{n}}{n!} \\
&=&Mte^{-tx+t^{2}x}<+\infty ,
\end{eqnarray*}%
such that DCT applies.\newline
\textbf{Under condition (b):} Now if $p_{n}a_{n}(t)\thicksim Ma^{n},$ for
some $M,a\in \Re $, then there exists constant $C>0$ such that%
\begin{equation*}
\left\vert p_{n}a_{n}(t)\right\vert \leq M\Rightarrow \left\vert
p_{n}n!/t^{n+1}\right\vert \leq Ma^{n}\Rightarrow \left\vert
p_{n}\right\vert \leq \frac{Ma^{n}t^{n+1}}{n!},
\end{equation*}%
and therefore%
\begin{equation*}
\left\vert f_{n}(x)\right\vert \leq \left\vert p_{n}\frac{\gamma ^{n}}{n!}%
x^{n}e^{-tx}\right\vert \leq \frac{\left\vert p_{n}\right\vert t^{n}}{n!}%
x^{n}e^{-tx}\leq \frac{Ma^{n}t^{2n+1}}{\left( n!\right) ^{2}}%
x^{n}e^{-tx}=h_{n}(x).
\end{equation*}%
Then it follows that
\begin{equation*}
\int\limits_{\Re }h_{n}(x)\mu _{1}(dx)=\int\limits_{0}^{+\infty }\frac{%
Ma^{n}t^{2n+1}}{\left( n!\right) ^{2}}x^{n}e^{-tx}\mu _{1}(dx)=\frac{%
Ma^{n}t^{2n+1}}{\left( n!\right) ^{2}}\int\limits_{0}^{+\infty
}x^{n}e^{-tx}\mu _{1}(dx)<+\infty ,
\end{equation*}%
and%
\begin{eqnarray*}
\sum\limits_{n\in \mathbb{N}}h_{n}(x) &=&\sum\limits_{n\in \mathbb{N}}\frac{%
Ma^{n}t^{2n+1}}{\left( n!\right) ^{2}}x^{n}e^{-tx}=Mte^{-tx}\sum\limits_{n%
\in \mathbb{N}}\frac{\left( at^{2}x\right) ^{n}}{\left( n!\right) ^{2}}\leq
Mte^{-tx}\sum\limits_{n\in \mathbb{N}}\frac{\left( at^{2}x\right) ^{n}}{n!}
\\
&=&Mte^{-tx+at^{2}x}<+\infty ,
\end{eqnarray*}%
such that DCT once again applies.\newline
\textbf{Under condition (c):} Define%
\begin{equation*}
L:=\underset{n\rightarrow +\infty }{\lim }\frac{\left\vert \gamma
p_{n+1}a_{n+1}(t)\right\vert }{(n+1)\left\vert p_{n}a_{n}(t)\right\vert }<1,
\end{equation*}%
and write
\begin{equation*}
\frac{|\gamma p_{n+1}a_{n+1}(t)|}{(n+1)|p_{n}a_{n}(t)|}=\frac{|\gamma
|\,|p_{n+1}|\dfrac{(n+1)!}{t^{n+2}}}{(n+1)|p_{n}|\dfrac{n!}{t^{n+1}}}=\frac{%
|\gamma |}{t}\frac{|p_{n+1}|}{|p_{n}|},
\end{equation*}%
such that
\begin{equation*}
L=\lim_{n\rightarrow \infty }\frac{|\gamma |}{t}\frac{|p_{n+1}|}{|p_{n}|}%
\Rightarrow \lim_{n\rightarrow \infty }\frac{|p_{n+1}|}{|p_{n}|}=\frac{tL}{%
|\gamma |}.
\end{equation*}%
Since $L<1$, we have
\begin{equation*}
\frac{tL}{|\gamma |}<\frac{t}{|\gamma |},
\end{equation*}%
so that we take $r\in \Re $ be any number satisfying
\begin{equation*}
\frac{tL}{|\gamma |}<r<\frac{t}{|\gamma |}.
\end{equation*}%
By the definition of the limit, there exists $N\in \mathbb{N}$ such that
\begin{equation*}
\frac{|p_{n+1}|}{|p_{n}|}<r,
\end{equation*}%
for all $n\geq N,$ and iterating this inequality gives
\begin{equation*}
\left\vert p_{n}\right\vert \leq \left\vert p_{N}\right\vert r^{n-N},
\end{equation*}%
for all $n\geq N.$ Thus, setting
\begin{equation*}
C:=\left\vert p_{N}\right\vert r^{-N},
\end{equation*}%
then
\begin{equation*}
\left\vert p_{n}\right\vert \leq Cr^{n},
\end{equation*}%
for all $n\geq N,$ and therefore, $p_{n}$ grows at most exponentially, or
more precisely
\begin{equation*}
\limsup_{n\rightarrow +\infty }\left\vert p_{n}\right\vert ^{\frac{1}{n}}=%
\frac{tL}{|\gamma |}<\frac{t}{|\gamma |}.
\end{equation*}%
We can rewrite the latter equivalently, using the standard definition of the
limit, i.e., for every $\varepsilon >0$, there exists a constant $%
C_{\varepsilon }>0$ such that
\begin{equation*}
\left\vert p_{n}\right\vert \leq C_{\varepsilon }\left( \frac{tL}{|\gamma |}%
+\varepsilon \right) ^{n}.
\end{equation*}%
Thus, for all $\varepsilon >0$ we have%
\begin{equation*}
\left\vert f_{n}(x)\right\vert \leq \left\vert p_{n}\frac{\gamma ^{n}}{n!}%
x^{n}e^{-tx}\right\vert \leq \frac{\left\vert p_{n}\right\vert t^{n}}{n!}%
x^{n}e^{-tx}\leq C_{\varepsilon }\left( \frac{tL}{|\gamma |}+\varepsilon
\right) ^{n}\frac{t^{n}}{n!}x^{n}e^{-tx}=h_{n}(x).
\end{equation*}%
Then it follows that
\begin{eqnarray*}
\int\limits_{\Re }h_{n}(x)\mu _{1}(dx) &=&\int\limits_{0}^{+\infty
}C_{\varepsilon }\left( \frac{tL}{|\gamma |}+\varepsilon \right) ^{n}\frac{%
t^{n}}{n!}x^{n}e^{-tx}\mu _{1}(dx) \\
&=&C_{\varepsilon }\left( \frac{tL}{|\gamma |}+\varepsilon \right) ^{n}\frac{%
t^{n}}{n!}\int\limits_{0}^{+\infty }x^{n}e^{-tx}\mu _{1}(dx)<+\infty ,
\end{eqnarray*}%
and%
\begin{eqnarray*}
\sum\limits_{n\in \mathbb{N}}h_{n}(x) &=&\sum\limits_{n\in \mathbb{N}%
}C_{\varepsilon }\left( \frac{tL}{|\gamma |}+\varepsilon \right) ^{n}\frac{%
t^{n}}{n!}x^{n}e^{-tx}=C_{\varepsilon }e^{-tx}\sum\limits_{n\in \mathbb{N}}%
\frac{\left( tx\left( \frac{tL}{|\gamma |}+\varepsilon \right) \right) ^{n}}{%
n!} \\
&=&C_{\varepsilon }e^{tx\left( \frac{tL}{|\gamma |}+\varepsilon -1\right)
}<+\infty ,
\end{eqnarray*}%
so that DCT can be applied in this case as well.\newline
\textbf{Applying the DCT:} After an appeal to Dominated Convergence Theorem,
we can write the TI $I_{\gamma ,\mathbf{a}(t)}^{\mathbf{p},\mathbb{N}}$\ as%
\begin{eqnarray*}
I_{\gamma ,\mathbf{a}(t)}^{\mathbf{p},\mathbb{N}} &=&\sum\limits_{n\in
\mathbb{N}}p_{n}a_{n}(t)\frac{\gamma ^{n}}{n!}=\sum\limits_{n\in \mathbb{N}%
}p_{n}\frac{\gamma ^{n}}{n!}\int\limits_{0}^{+\infty }x^{n}e^{-tx}\mu
_{1}(dx) \\
&=&\int\limits_{0}^{+\infty }\sum\limits_{n\in \mathbb{N}}p_{n}\frac{\gamma
^{n}}{n!}x^{n}e^{-tx}\mu _{1}(dx)=\int\limits_{0}^{+\infty }\left[
\sum\limits_{n\in \mathbb{N}}p_{n}\frac{(\gamma x)^{n}}{n!}\right]
e^{-tx}\mu _{1}(dx),
\end{eqnarray*}%
i.e., the Laplace transform of the Taylor measure $T_{\gamma x,\mathbf{p}}(%
\mathbb{N})$ is given by%
\begin{equation*}
I_{\gamma ,\mathbf{a}(t)}^{\mathbf{p},\mathbb{N}}=\int\limits_{0}^{+\infty
}T_{\gamma x,\mathbf{p}}(\mathbb{N})e^{-tx}dx.
\end{equation*}%
By the assumption $I_{\gamma ,\mathbf{a}(t)}^{\mathbf{p},\mathbb{N}%
}=I_{\gamma ,\mathbf{a}(t)}^{\mathbf{q},\mathbb{N}},$ for all $t>0,$ and
uniqueness of the Laplace transform we must have%
\begin{equation*}
T_{\gamma x,\mathbf{p}}(\mathbb{N})=T_{\gamma x,\mathbf{q}}(\mathbb{N}),
\end{equation*}%
for all $x>0$. Thus%
\begin{equation*}
\sum\limits_{n\in \mathbb{N}}p_{n}\frac{\gamma ^{n}}{n!}x^{n}=\sum\limits_{n%
\in \mathbb{N}}q_{n}\frac{\gamma ^{n}}{n!}x^{n},
\end{equation*}%
and rewriting the latter as a series in $x$, we have%
\begin{equation*}
\sum\limits_{n\in \mathbb{N}}\left[ p_{n}\frac{\gamma ^{n}}{n!}-q_{n}\frac{%
\gamma ^{n}}{n!}\right] x^{n}=0,
\end{equation*}%
for all $x>0$, which leads to%
\begin{equation*}
p_{n}\frac{\gamma ^{n}}{n!}=q_{n}\frac{\gamma ^{n}}{n!}\Rightarrow
p_{n}=q_{n},
\end{equation*}%
for all $n\in \mathbb{N}.$
\end{proof}

In view of the latter theorem, the Taylor measure defined by%
\begin{equation*}
\mathcal{T}_{\gamma ,t}(B):=T_{\gamma ,\mathbf{a}(t)}(B)=\sum\limits_{n\in B}%
\frac{n!}{t^{n+1}}\frac{\gamma ^{n}}{n!}=\frac{1}{t}\sum\limits_{n\in
B}(\gamma /t)^{n},
\end{equation*}%
with%
\begin{equation*}
\mathcal{T}_{\gamma ,t}(\mathbb{N})=\frac{1}{t}\sum\limits_{n\in \mathbb{N}%
}(\gamma /t)^{n}=\frac{t}{t-\gamma }\frac{1}{t}=\frac{1}{t-\gamma },
\end{equation*}%
where $a_{n}(t)=n!/t^{n+1},$ $n\in \mathbb{N},$ for some $\gamma \neq 0,\ $%
provided that $\left\vert t\right\vert >\left\vert \gamma \right\vert ,$
emerges as of particular importance, since it acts as the integrating
measure required by the TI to identify uniqueness, in a similar fashion as
the measure $e^{-tx}\mu _{1}(dx)$ is required by the Laplace transform or $%
e^{-itx}\mu _{1}(dx)$ for the Fourier transform. The latter observation
leads to the following definition.

\begin{definition}[Discrete Taylor Transformation]
The discrete Taylor transformation\newline
(DTT) of a sequence\label{TMITransform} $s_{n}:\mathbb{N}\rightarrow \mathbb{%
M},$ is given by the Taylor integral of $\mathbf{s}$\ with respect to $%
\mathcal{T}_{\gamma ,t},$ defined by%
\begin{equation}
\mathcal{T}_{\gamma ,B}^{s_{n}}(t):=\int\limits_{B}s_{n}\mathcal{T}_{\gamma
,t}(dn)=\frac{1}{t}\sum\limits_{n\in B}s_{n}\frac{\gamma ^{n}}{t^{n}},
\label{DTT1}
\end{equation}%
for any $B\in \mathcal{B}(\mathbb{N}),$ $\gamma \neq 0,$ and for all $t\in
\mathbb{M},$ such that $\left\vert t\right\vert >\left\vert \gamma
\right\vert .$ When $B=\mathbb{N}$, we drop the $B$ from the DTT notation
and write%
\begin{equation}
\mathcal{T}_{\gamma }^{s_{n}}(t):=\int\limits_{\mathbb{N}}s_{n}\mathcal{T}%
_{\gamma ,t}(dn)=\frac{1}{t}\sum\limits_{n\in \mathbb{N}}s_{n}\frac{\gamma
^{n}}{t^{n}}.  \label{DTT2}
\end{equation}%
The DTT for finite $B$ exists always, and for countably infinite $B$ it
exists and is finite provided that one of the following conditions hold:%
\newline
a) $s_{n}$ are uniformly bounded, i.e., $|s_{n}|\leq M,$ $\forall n\in
\mathbb{N}$, for some $M>0,$ or\newline
b) $s_{n}$ are asymptotically equivalent to $Ma^{n}$, i.e., $s_{n}\thicksim
Ma^{n},$ for some $M,a\in \Re $, or\newline
c) if $s_{n}\neq 0,$ $\forall n>n_{0}\in \mathbb{N},$ then $\underset{%
n\rightarrow +\infty }{\lim }\frac{\left\vert \gamma s_{n+1}\right\vert }{%
\left\vert s_{n}\right\vert }<1,$ guarantees absolute convergence.\newline
In addition, if $\gamma _{k}\neq 0,$ $\forall k\in B,$ is a sequence (or
vector) the DTT $\mathcal{T}_{\gamma _{k},B}^{s_{n}}(t),$ defines a new
sequence (or vector) of measurable functions $\{\mathcal{T}_{\gamma
_{k,B}}^{s_{n}}(t)\}_{k\in B}$ which will be called the DTT of the sequence $%
\mathbf{s}$ driven by the sequence (or vector) $\mathbf{\gamma }=(\gamma
_{0},\gamma _{1},...).$
\end{definition}

Note that the DTT is related to the OGF $G$\ of a real sequence $\mathbf{s},$
via%
\begin{equation}
\mathcal{T}_{\gamma }^{s_{n}}(t)=\frac{1}{t}G(\gamma /t)\Leftrightarrow G(x)=%
\frac{\gamma }{x}\mathcal{T}_{\gamma }^{s_{n}}(\frac{\gamma }{x})=t\mathcal{T%
}_{\gamma }^{s_{n}}(t),  \label{DTTtoOGF}
\end{equation}%
for $x=\gamma /t\neq 0$, and therefore this important tool for sequences is
a special case of the DTT. Importantly, the DTT converts mathematical
results about sequences into problems about functions, allowing us to use
function manipulation to solve sequence-related problems, including finding
closed form solutions to difference equations, as well as manipulate a
series without needing to worry about whether the series converges for a
specific value of $t$.

For the general domain of convergence of the DTT $\mathcal{T}_{\gamma
}^{s_{n}}(t)$, we note that given the radius of convergence $R_{s}$ of the
generating function%
\begin{equation*}
G_{s}(z)=\sum_{n=0}^{\infty }s_{n}z^{n},
\end{equation*}%
such that $\mathcal{T}_{\gamma }^{s_{n}}(t)$ exists whenever%
\begin{equation*}
\left\vert \frac{\gamma }{t}\right\vert <R_{s}\Rightarrow \left\vert
t\right\vert >\frac{\left\vert \gamma \right\vert }{R_{s}},
\end{equation*}%
and in terms of $\mathcal{T}_{\gamma _{k},B}^{s_{n}}(t),$ we have that $t\in
\mathbb{T}_{B},$ where%
\begin{equation}
\mathbb{T}_{B}=\left\{ t\in \mathbb{M}:\left\vert t\right\vert >\frac{%
\underset{k\in B}{\max }\left\vert \gamma _{k}\right\vert }{R_{s}}\right\} ,
\label{Tspace1}
\end{equation}%
For $B=\mathbb{N},$ the DTT sequence thus created is $\{\mathcal{T}_{\gamma
_{k}}^{s_{n}}(t)\}_{k\in \mathbb{N}}$ and we write $\mathbb{T}$ for $\mathbb{%
T}_{\mathbb{N}}.$

\subsection{Properties of the DTT}

We collect properties of DTTs next.

\begin{remark}[Properties of $\mathcal{T}_{\protect\gamma }^{f_{n}}(t)$]
\label{PropsDTTs}Consider any $B\in \mathcal{B}(\mathbb{N}),$ and any
sequences $f_{n},g_{n}:$ $\mathbb{N}\rightarrow \Re ,$ with $\mathbf{f}%
=(f_{0},f_{1},...)$ and $\mathbf{g}=(g_{0},g_{1},...).$ Unless otherwise
stated, in all the statements that follow we assume that they hold for all $%
\left\vert t\right\vert >\left\vert \gamma \right\vert .$ Note that
everything that follows holds for DTTs when the constant $\gamma $ is
replaced by a sequence $\gamma _{k}\neq 0,$ and keeping $k$ fixed, with $%
t\in \mathbb{T}$. There are many other properties that are derived via
results on generating functions that will be omitted. In addition, since the
DTT is a TI, it satisfies all properties of Remark \ref{PropsTMI}.

\begin{enumerate}
\item Uniqueness: In view of Theorem \ref{UniqDTT}, we have%
\begin{equation}
\mathcal{T}_{\gamma }^{f_{n}}(t)=\mathcal{T}_{\gamma }^{g_{n}}(t),\forall
\left\vert t\right\vert >\left\vert \gamma \right\vert \Rightarrow
f_{n}=g_{n},\forall n\in \mathbb{N},  \label{DTTUniq1}
\end{equation}%
for $\gamma \neq 0.$

\item Differentiation: Assuming we can swap the differentiation and
summation signs (i.e., the sum of derivatives converges), we have%
\begin{equation*}
\frac{d}{dt}\mathcal{T}_{\gamma ,B}^{f_{n}}(t)=\sum\limits_{n\in B}s_{n}%
\frac{d}{dt}\left( \frac{\gamma ^{n}}{t^{n}}\right) =-\sum\limits_{n\in
B}ns_{n}\frac{\gamma ^{n}}{t^{n+1}}=-\frac{1}{t}\sum\limits_{n\in B}ns_{n}%
\frac{\gamma ^{n}}{t^{n}}=-\mathcal{T}_{\gamma ,B}^{nf_{n}}(t).
\end{equation*}

\item Scaling: For any $a\neq 0,$ we have%
\begin{equation*}
\mathcal{T}_{\gamma ,B}^{\mathbf{f}}(t/a)=\frac{a}{t}\sum\limits_{n\in
B}a^{n}f_{n}\frac{\gamma ^{n}}{t^{n}}=\frac{a}{t}\sum\limits_{n\in B}f_{n}%
\frac{(a\gamma )^{n}}{t^{n}}=a\mathcal{T}_{a\gamma ,B}^{\mathbf{f}}(t).
\end{equation*}

\item Binomial convolution: The DTT of the convolution of $\mathbf{f}$ and $%
\mathbf{g}$ defined by $p_{n}=\sum\limits_{k=0}^{n}\mathcal{C}_{k}^{n}$ $%
f_{k}g_{n-k},$ is expressed as a component-wise product of the DTTs $%
\mathcal{T}_{\gamma }^{f_{n}}(t)$ and $\mathcal{T}_{\gamma }^{g_{n}}(t)$ via%
\begin{eqnarray*}
\mathcal{T}_{\gamma }^{f_{n}}(t)\mathcal{T}_{\gamma }^{g_{n}}(t) &=&\frac{1}{%
t^{2}}\sum\limits_{n\in \mathbb{N}}f_{n}\frac{\gamma ^{n}}{t^{n}}%
\sum\limits_{k\in \mathbb{N}}g_{k}\frac{\gamma ^{k}}{t^{k}}=\frac{1}{t^{2}}%
\sum\limits_{n\in \mathbb{N}}\sum\limits_{k=0}^{n}\mathcal{C}_{k}^{n}f_{k}%
\frac{\gamma ^{k}}{t^{k}}g_{n-k}\frac{\gamma ^{n-k}}{t^{n-k}} \\
&=&\frac{1}{t^{2}}\sum\limits_{n\in \mathbb{N}}\frac{\gamma ^{n}}{t^{n}}%
\sum\limits_{k=0}^{n}\mathcal{C}_{k}^{n}f_{k}g_{n-k}=\frac{1}{t^{2}}%
\sum\limits_{n\in \mathbb{N}}p_{n}\frac{\gamma ^{n}}{t^{n}},
\end{eqnarray*}%
and therefore%
\begin{equation*}
\mathcal{T}_{\gamma }^{f_{n}}(t)\mathcal{T}_{\gamma }^{g_{n}}(t)=\frac{1}{t}%
\mathcal{T}_{\gamma }^{p_{n}}(t).
\end{equation*}

\item Inversion via differentiation: Given a DTT $\mathcal{T}_{\gamma
}^{f_{n}}(t)$ we can recover the elements of the sequence $\mathbf{f}$ as
follows; since the DTT is a generator function, consider%
\begin{equation}
P_{\mathbf{f}}(t):=\gamma \mathcal{T}_{\gamma }^{f_{n}}(\gamma
/t)=t\sum\limits_{n=0}^{+\infty }f_{n}t^{n}=f_{0}t+f_{1}t^{2}+f_{2}t^{3}+...,
\label{Poly1}
\end{equation}%
and take derivative with respect to $t$ repeatedly to obtain%
\begin{equation}
f_{k}=\frac{1}{k!}\left[ \frac{d^{k+1}}{dt^{k+1}}\left( \gamma \mathcal{T}%
_{\gamma }^{f_{n}}(\gamma /t)\right) \right] _{t=0}=\frac{1}{k!}\left[ \frac{%
d^{k+1}}{dt^{k+1}}P_{\mathbf{f}}(t)\right] _{t=0},  \label{Inverse1}
\end{equation}%
for all $k\in \mathbb{N}$.

\item Inversion using Cauchy's coefficient formula: Consider a simple,
nonintersecting closed curve $\mathcal{C}\subset \mathbb{C}$ traversed
counterclockwise, and recall Cauchy's integral equation (e.g., \cite%
{keener2018principles}, Chapter 6). Suppose that $f(z)$ is analytic
everywhere inside $\mathcal{C}$. Then for any point $z$ inside $\mathcal{C}$
we can write%
\begin{equation}
f(z)=\frac{1}{2\pi i}\oint\limits_{\mathcal{C}}\frac{f(\xi )}{\xi -z}d\xi ,
\label{Cauchy1}
\end{equation}%
with the $n$th derivative written as%
\begin{equation}
f^{(n)}(z)=\frac{n!}{2\pi i}\oint\limits_{\mathcal{C}}\frac{f(\xi )}{\left(
\xi -z\right) ^{n+1}}d\xi .  \label{Cauchy2}
\end{equation}%
Now take $f(z)=P_{\mathbf{f}}(z),$ so that Equation (\ref{Cauchy2}) yields
an equivalent form of (\ref{Inverse1}) given by%
\begin{equation*}
f_{k}=\frac{1}{k!}P_{\mathbf{f}}^{(k+1)}(0)=\frac{1}{2\pi i}\oint\limits_{%
\mathcal{C}}\frac{P_{\mathbf{f}}(\xi )}{\xi ^{k+2}}d\xi .
\end{equation*}%
If the curve $\mathcal{C}$ is a circle, $\xi =z+re^{i\theta },$ Equation (%
\ref{Cauchy1}) leads to Gauss' Mean-Value Theorem%
\begin{equation*}
P_{\mathbf{f}}(z)=\frac{1}{2\pi }\int\limits_{0}^{2\pi }P_{\mathbf{f}%
}(z+re^{i\theta })d\theta ,
\end{equation*}%
so that in terms of the DTT we have%
\begin{equation*}
\mathcal{T}_{\gamma }^{f_{n}}(\gamma /z)=\frac{1}{2\pi }\int\limits_{0}^{2%
\pi }\mathcal{T}_{\gamma }^{f_{n}}\left( \frac{\gamma }{z+re^{i\theta }}%
\right) d\theta .
\end{equation*}%
Moreover, the inversion equation over the circle $re^{i\theta },$ in terms
of the DTT, becomes%
\begin{equation}
f_{k}=\frac{\gamma }{2\pi r^{k+1}}\int\limits_{0}^{2\pi }\mathcal{T}_{\gamma
}^{f_{n}}\left( \frac{\gamma }{r}e^{-i\theta }\right) e^{-(k+1)i\theta
}d\theta .  \label{Inverse2}
\end{equation}

\item Hadamard product: The dimension-wise product $\mathbf{f}\cdot \mathbf{g%
}=[f_{0}g_{0},f_{1}g_{1},...]$ of two sequences has DTT given by%
\begin{equation*}
\mathcal{T}_{\gamma }^{f_{n}g_{n}}(t)=\frac{1}{t}\sum\limits_{n\in \mathbb{N}%
}f_{n}g_{n}\frac{\gamma ^{n}}{t^{n}},
\end{equation*}%
and it is related to the DTTs $\mathcal{T}_{\gamma }^{f_{n}}(t)$ and $%
\mathcal{T}_{\gamma }^{g_{n}}(t),$ as follows; using Equation (\ref{Inverse1}%
) we can write%
\begin{eqnarray*}
\mathcal{T}_{\gamma }^{f_{n}g_{n}}(t) &=&\frac{1}{t}\sum\limits_{n\in
\mathbb{N}}\frac{1}{\left( n!\right) ^{2}}\left[ \frac{d^{n+1}}{du^{n+1}}P_{%
\mathbf{f}}(u)\right] _{u=0}\left[ \frac{d^{n+1}}{dv^{n+1}}P_{\mathbf{g}}(v)%
\right] _{v=0}\frac{\gamma ^{n}}{t^{n}} \\
&=&\mathcal{T}_{\gamma }^{\frac{1}{\left( n!\right) ^{2}}\left[ \frac{d^{n+1}%
}{du^{n+1}}P_{\mathbf{f}}(u)\right] _{u=0}\left[ \frac{d^{n+1}}{dv^{n+1}}P_{%
\mathbf{g}}(v)\right] _{v=0}}(t).
\end{eqnarray*}%
In terms of the OGFs $F$ and $G$ of the sequences $\mathbf{f}$ and $\mathbf{g%
},$ the Hadamard product $G_{1}\odot G_{2}$ is defined by%
\begin{equation*}
(G_{1}\odot G_{2})(z):=\sum\limits_{n\in \mathbb{N}}f_{n}g_{n}z^{n}=\frac{1}{%
2\pi }\int\limits_{0}^{2\pi }G_{1}\left( \sqrt{z}e^{iu}\right) G_{2}\left(
\sqrt{z}e^{-iu}\right) \mu _{1}(du),
\end{equation*}%
where $z\in \mathbb{C}$, and using (\ref{DTTtoOGF}), we have%
\begin{equation*}
(G_{1}\odot G_{2})(z)=\sum\limits_{n\in \mathbb{N}}f_{n}g_{n}z^{n}=\frac{%
\gamma }{z}\mathcal{T}_{\gamma }^{f_{n}g_{n}}\left( \frac{\gamma }{z}\right)
,
\end{equation*}%
$\gamma \in \Re ,$ and therefore the Hadamard product is a special case of
the DTT. Because $G_{1}(z)=\frac{\gamma }{z}\mathcal{T}_{\gamma }^{f_{n}}(%
\frac{\gamma }{z}),$ and $G_{2}(z)=\frac{\gamma }{z}\mathcal{T}_{\gamma
}^{g_{n}}(\frac{\gamma }{z}),$ the DTT of the product of two sequences is
given by%
\begin{equation}
\mathcal{T}_{\gamma }^{f_{n}g_{n}}(t)=\frac{\gamma }{2\pi }%
\int\limits_{0}^{2\pi }\mathcal{T}_{\gamma }^{f_{n}}(\sqrt{\gamma t}e^{-iu})%
\mathcal{T}_{\gamma }^{g_{n}}(\sqrt{\gamma t}e^{iu})\mu _{1}(du),
\label{ProductDTT}
\end{equation}%
where $t=\gamma /z\in \mathbb{C}$, and as a result, the DTT of the square of
a sequence is given by%
\begin{equation}
\mathcal{T}_{\gamma }^{f_{n}^{2}}(t)=\frac{1}{t}\sum\limits_{n\in \mathbb{N}%
}f_{n}^{2}\frac{\gamma ^{n}}{t^{n}}=\frac{\gamma }{2\pi }\int\limits_{0}^{2%
\pi }\mathcal{T}_{\gamma }^{f_{n}}(\sqrt{\gamma t}e^{-iu})\mathcal{T}%
_{\gamma }^{f_{n}}(\sqrt{\gamma t}e^{iu})\mu _{1}(du).  \label{SquareDTT}
\end{equation}

\item Jensen-Mercer inequality: For any measurable function $h,$ the DTT of
the sequence $h(f_{n})/n!$ is connected to a Taylor measure via%
\begin{equation*}
T_{\gamma /t,\mathbf{h}/t}(B)=\sum\limits_{n\in B}\frac{h(f_{n})}{t}\frac{%
(\gamma /t)^{n}}{n!}=\frac{1}{t}\sum\limits_{n\in B}\frac{h(f_{n})}{n!}\frac{%
\gamma ^{n}}{t^{n}}=\mathcal{T}_{\gamma ,B}^{h(f_{n})/n!}(t),
\end{equation*}%
where $\mathbf{h}=(h(f_{0}),h(f_{1}),h(f_{2}),...)$ and $\mathbf{h}%
/t:=(h(f_{0})/t,h(f_{1})/t,h(f_{2})/t,...)$. Moreover, the transformed
sequence $h(f_{n})$ has DTT given by%
\begin{equation*}
\mathcal{T}_{\gamma ,B}^{h(f_{n})}(t)=\frac{1}{t}\sum\limits_{n\in B}h(f_{n})%
\frac{\gamma ^{n}}{t^{n}}=\frac{1}{t}\sum\limits_{n\in B}n!h(f_{n})\frac{%
(\gamma /t)^{n}}{n!}=\frac{1}{t}T_{\gamma /t,\mathbf{n!}\cdot \mathbf{h}}(B).
\end{equation*}%
When $h$ is convex, and $\mathbf{f}$ is integrable $[\mathcal{T}_{\gamma
,t}] $, taking $B=\mathbb{N}$, an appeal to the signed measure theoretic
version of the Jensen-Mercer inequality (see \cite{mercer2003variant}, \cite%
{horvath2024refining}, and \cite{horvath2025integral} for details and
specific versions of the inequality, along with all the conditions
required), allows us to write%
\begin{equation}
h\left( a+b-\frac{1}{\mathcal{T}_{\gamma ,t}(\mathbb{N})}\mathcal{T}_{\gamma
,B}^{f_{n}}(t)\right) \leq h(a)+h(b)-\frac{1}{\mathcal{T}_{\gamma ,t}(%
\mathbb{N})}\mathcal{T}_{\gamma ,B}^{h(f_{n})}(t),  \label{Jensen}
\end{equation}%
provided that $\mathcal{T}_{\gamma ,t}(\mathbb{N})>0,$ and $f_{n}\in \lbrack
a,b]$, for all $n\in \mathbb{N}.$

\item Matrix Representation: Take $B_{N}=\{0,1,...,N-1\}\in \mathcal{B}(%
\mathbb{N}).$ The DTT $\mathcal{T}_{\gamma _{k},B_{N}}^{f_{n}}(1)=$ $%
\sum\limits_{n=0}^{N-1}f_{n}\gamma _{k}^{n},$ $k\in B_{N}$, of the sequence $%
\mathbf{f}$ driven by the sequence $\mathbf{\gamma }=(\gamma _{0},$ $\gamma
_{1},...),$ at $t=1$, can be expressed in terms of matrices as follows;
define the Vandermonde matrix%
\begin{equation}
\mathbf{\Gamma }_{N}=\mathbb{V}(\gamma _{0},...,\gamma _{N-1}):=\left[
\begin{tabular}{lllll}
$\gamma _{0}^{0}$ & $\gamma _{0}^{1}$ & $...$ & $\gamma _{0}^{N-2}$ & $%
\gamma _{0}^{N-1}$ \\
$\gamma _{1}^{0}$ & $\gamma _{1}^{1}$ & $...$ & $\gamma _{1}^{N-2}$ & $%
\gamma _{1}^{N-1}$ \\
$...$ & $...$ &  & $...$ & $...$ \\
$\gamma _{N-1}^{0}$ & $\gamma _{N-1}^{1}$ & $...$ & $\gamma _{N-1}^{N-2}$ & $%
\gamma _{N-1}^{N-1}$%
\end{tabular}%
\right] ,  \label{Gammamat}
\end{equation}%
with determinant $\det (\mathbf{\Gamma }_{N})=\prod\limits_{0\leq i<j\leq
N-1}(\gamma _{j}-\gamma _{i})$ (which is non-zero for distinct $\mathbf{%
\gamma }$), so that the sequence $\mathbf{g}=\left( \mathcal{T}_{\gamma
_{0},B_{N}}^{f_{n}}(1),\mathcal{T}_{\gamma _{1},B_{N}}^{f_{n}}(1),...,%
\mathcal{T}_{\gamma _{N-1},B_{N}}^{f_{n}}(1),...\right) ,$ can be written as%
\begin{equation}
\mathbf{g}_{N}=\mathbf{\Gamma }_{N}\mathbf{f}_{N},  \label{MatRepDTT}
\end{equation}%
where $\mathbf{g}_{N}$ and $\mathbf{f}_{N}$ are column vectors corresponding
to the first $N$ elements of the sequences $\mathbf{g}$ and $\mathbf{f}$,
respectively, and $N\geq 1.$
\end{enumerate}
\end{remark}

We present some mathematical applications of DTTs, in particular, to solving
difference equations. Other applications of the DTT are presented in Section %
\ref{Applications}.

\begin{example}[Difference Equations]
\label{AppDTTs}In what follows, let $f_{n}:$ $\mathbb{N}\rightarrow \mathbb{M%
},$ with $\mathbf{f}=(f_{0},f_{1},...),$ be any sequence.

\begin{enumerate}
\item First order difference equations: Consider a general first order
difference equation of the form%
\begin{equation}
f_{n}=a_{n}f_{n-1}+b_{n},  \label{FirstODE}
\end{equation}%
for all $n\geq 1,$ for a given starting value $f_{0},$ and known sequences $%
a_{n}$ and $b_{n},$ where we set $a_{0}=b_{0}=0.$ Then multiplying with $%
\frac{\gamma ^{n}}{t^{n}}$ both sides and summing over $n$ we obtain%
\begin{equation*}
\mathcal{T}_{\gamma }^{f_{n}}(t)=\frac{f_{0}}{t}+\frac{\gamma }{t}\mathcal{T}%
_{\gamma }^{\frac{a_{n+1}}{n!}f_{n}}(t)+\mathcal{T}_{\gamma }^{b_{n}}(t),
\end{equation*}%
$n\geq 1,$ which can be solved to provide the general solution of (\ref%
{FirstODE}), for given DTTs $\mathcal{T}_{\gamma }^{a_{n}}(t)$ and $\mathcal{%
T}_{\gamma }^{b_{n}}(t)$, provided the form of $\mathcal{T}_{\gamma }^{\frac{%
a_{n+1}}{n!}f_{n}}(t)$ is available.

\item Second order difference equations: Now turn to a general second order
difference equation of the form%
\begin{equation}
f_{n}=a_{n}f_{n-1}+b_{n}f_{n-2}+c_{n},  \label{SecondODE}
\end{equation}%
for all $n\geq 2,$ for given starting values $f_{0}$ and $f_{1},$ and known
sequences $a_{n},$ $b_{n}$ and $c_{n},$ where we set $%
a_{0}=b_{0}=c_{0}=a_{1}=b_{1}=c_{1}=0.$ Multiplying with $\frac{\gamma ^{n}}{%
t^{n}}$ both sides and summing over $n$ we obtain%
\begin{equation*}
\mathcal{T}_{\gamma }^{f_{n}}(t)=\frac{f_{0}}{t}+f_{1}\frac{\gamma }{t^{2}}+%
\frac{\gamma }{t}\mathcal{T}_{\gamma }^{\frac{a_{n+1}}{n!}f_{n}}(t)+\frac{%
\gamma ^{2}}{t^{2}}\mathcal{T}_{\gamma }^{\frac{b_{n+2}}{n!}f_{n}}(t)+%
\mathcal{T}_{\gamma }^{c_{n}}(t),
\end{equation*}%
$n\geq 2,$ and for given DTTs $\mathcal{T}_{\gamma }^{a_{n}}(t),$ $\mathcal{T%
}_{\gamma }^{b_{n}}(t),$ and $\mathcal{T}_{\gamma }^{c_{n}}(t)$, the latter
can be solved to provide the general solution of (\ref{SecondODE}), provided
the forms of $\mathcal{T}_{\gamma }^{\frac{a_{n+1}}{n!}f_{n}}(t)$ and $%
\mathcal{T}_{\gamma }^{\frac{b_{n+2}}{n!}f_{n}}(t)$ are available.

\item Fibonacci sequence: Take $f_{0}=0,$ $f_{1}=1,$ and $%
f_{n}=f_{n-1}+f_{n-2},$ $n\geq 2,$ i.e., $\mathbf{f}%
=(0,1,1,2,3,5,8,13,21,...),$ so that multiplying with $\frac{\gamma ^{n}}{%
t^{n}}$ and summing over $n$ we have that the Fibonacci DTT is given by%
\begin{equation*}
\mathcal{T}_{\gamma }^{f_{n}}(t)=\frac{\gamma }{t^{2}-\gamma t-\gamma ^{2}},
\end{equation*}%
for $\left\vert t\right\vert >\frac{1+\sqrt{5}}{2}\left\vert \gamma
\right\vert ,$ since the Fibonacci generating function has radius of
convergence $\frac{2}{1+\sqrt{5}}.$

\item Catalan numbers: Many combinatorial problems have as their solution
the Catalan numbers given by $C_{n}=\frac{1}{n+1}\frac{(2n)!}{n!n!},$ $n\in
\mathbb{N},$ which satisfy the Catalan recurrence equation $%
C_{n+1}=\sum\limits_{k=0}^{n}C_{k}C_{n-k}$. This is a convolution of the
sequence $C_{n}$ with itself, and it is straightforward to show that the
Catalan DTT is given by%
\begin{equation*}
\mathcal{T}_{\gamma }^{C_{n}}(t)=\frac{1-\sqrt{1-4\gamma /t}}{2\gamma },
\end{equation*}%
for $\left\vert t\right\vert >4\left\vert \gamma \right\vert .$
\end{enumerate}
\end{example}

\subsection{DTT Inversion and DFT as special case}

Assume that $B\in \mathcal{B}(\mathbb{N}),$ is nonempty with cardinality $%
N=\#(B)>0$, and consider two sequences $f_{n},\gamma _{n}:$ $\mathbb{N}%
\rightarrow \mathbb{M},$ with $\mathbf{f}=(f_{0},f_{1},...),$ and $\mathbf{%
\gamma }$ $=$ $(\gamma _{0},\gamma _{1},...),$ $\gamma _{n}\neq 0,$ $\forall
n\in B,$ and $f_{n}\neq 0,$ for at least one $n\in B.$ The DTT of $\mathbf{f}
$ driven by the sequence $\mathbf{\gamma },$ defines the sequence of
functions%
\begin{equation}
g_{k}(t)=\mathcal{T}_{\gamma _{k},B}^{f_{n}}(t)=\int\limits_{B}f_{n}\mathcal{%
T}_{\gamma _{k},t}(dn)=\frac{1}{t}\sum\limits_{n\in B}f_{n}\frac{\gamma
_{k}^{n}}{t^{n}},  \label{DTTfdrivengamma}
\end{equation}%
for all $k\in B$, $t\in \mathbb{T}_{B},$ where $g_{k}(t):\mathbb{T}\times
B\rightarrow \mathbb{M},$ with $\mathbf{g}(t)=$ $(g_{0}(t),g_{1}(t),...),$
is a sequence of measurable functions by construction.

The inverse DTT (IDTT) of the sequence $\mathbf{g}$ driven by a sequence $%
\mathbf{\xi }=$ $(\xi _{0},\xi _{1},...)$, is defined by the DTT%
\begin{equation}
\mathcal{IT}_{\xi _{n},B}^{g_{k}}(t):=\frac{1}{w_{n}(\mathbf{\gamma },N)}%
\int\limits_{B}g_{k}(t)\mathcal{T}_{\xi _{k},t}(dk)=\frac{1}{tw_{n}(\mathbf{%
\gamma },N)}\sum\limits_{k\in B}g_{k}(t)\frac{\xi _{n}^{k}}{t^{k}},
\label{IDTT1}
\end{equation}%
where $w_{n}(\mathbf{\gamma },N)\neq 0,$ a (weight) function of the
cardinality of $B$ and the sequence $\mathbf{\gamma }$, and $\mathbf{\xi }$
is a sequence that depends on $\mathbf{\gamma },$ and is such that at $t=1$,
we recover the original sequence, i.e.,
\begin{equation}
f_{n}=\mathcal{IT}_{\xi _{n},B}^{g_{k}}(1),  \label{IDTT2}
\end{equation}%
for all $n\in B$. The following provides conditions in order for the DTT $%
\mathcal{T}_{\gamma _{k},B}^{f_{n}}(t)$ of Equation (\ref{DTTfdrivengamma})
to be uniquely invertible and such that the IDTT\ $\mathcal{IT}_{\xi
_{n},B}^{g_{k}}(t)$ of Equation (\ref{InvDTT1}) to be itself a DTT. In other
words, we find conditions under which the inverse Vandermonde matrix, after
diagonal scaling, is itself a Vandermonde matrix.\textquotedblright

\begin{theorem}[DTT Inversion]
\label{DTTInversion}Given the non-zero and distinct driving sequence $%
\mathbf{\gamma }$, the DTT $g_{k}(t)=\mathcal{T}_{\gamma _{k},B}^{f_{n}}(t)$
of an arbitrary sequence $\mathbf{f}$, is uniquely invertible with $f_{n}$
given by (\ref{IDTT2}) if and only if the sequence $\mathbf{\xi }$ satisfies
the orthogonality conditions%
\begin{equation}
\sum\limits_{k\in B}\gamma _{k}^{l}\xi _{n}^{k}=w_{n}(\mathbf{\gamma }%
,N)\delta _{nl},  \label{IDTTCond}
\end{equation}%
where $\delta _{nl}$ denotes Kronecker's delta function, for all $n,l\in B$.
\end{theorem}

\begin{proof}
\textbf{Form of }$\mathcal{IT}_{\xi _{n},B}^{g_{k}}(t)$\textbf{\ : }By
definition of $\mathcal{IT}_{\xi _{n},B}^{g_{k}}(t)$ and $g_{k}(t)$, we can
write%
\begin{eqnarray*}
\mathcal{IT}_{\xi _{n},B}^{g_{k}}(t) &=&\frac{1}{w_{n}(\mathbf{\gamma },N)}%
\int\limits_{B}g_{k}(t)\mathcal{T}_{\xi _{n},t}^{B}(dk)=\frac{1}{tw_{n}(%
\mathbf{\gamma },N)}\sum\limits_{k\in B}g_{k}(t)\frac{\xi _{n}^{k}}{t^{k}} \\
&=&\frac{1}{tw_{n}(\mathbf{\gamma },N)}\sum\limits_{k\in B}\left[ \frac{1}{t}%
\sum\limits_{l\in B}f_{l}\frac{\gamma _{k}^{l}}{t^{l}}\right] \frac{\xi
_{n}^{k}}{t^{k}} \\
&=&\frac{1}{t^{2}w_{n}(\mathbf{\gamma },N)}\sum\limits_{k\in
B}\sum\limits_{l\in B}f_{l}\frac{\gamma _{k}^{l}}{t^{l}}\frac{\xi _{n}^{k}}{%
t^{k}}=\frac{1}{t^{2}w_{n}(\mathbf{\gamma },N)}\sum\limits_{l\in B}f_{l}%
\frac{1}{t^{l}}\sum\limits_{k\in B}\gamma _{k}^{l}\frac{\xi _{n}^{k}}{t^{k}}.
\end{eqnarray*}%
For $t=1,$ we have%
\begin{equation}
\mathcal{IT}_{\xi _{n},B}^{g_{k}}(1)=\frac{1}{w_{n}(\mathbf{\gamma },N)}%
\sum\limits_{k\in B}\sum\limits_{l\in B}f_{l}\gamma _{k}^{l}\xi _{n}^{k}=%
\frac{1}{w_{n}(\mathbf{\gamma },N)}\sum\limits_{l\in
B}f_{l}\sum\limits_{k\in B}\gamma _{k}^{l}\xi _{n}^{k}.  \label{InvDTT1}
\end{equation}%
\newline
$(\Rightarrow )$ Assume that the driving sequence $\mathbf{\gamma }$ is
non-zero and distinct, and that (\ref{IDTT2}) holds, and note that the
matrix $\left[ \left( \xi _{n}^{k}\right) \right] $ is not Vandermonde for
general $B$. Let $\zeta _{nl}=\sum\limits_{k\in B}\gamma _{k}^{l}\xi
_{n}^{k},$ for all $n,l\in B$, so that (\ref{InvDTT1}) yields%
\begin{equation*}
f_{n}=\frac{1}{w_{n}(\mathbf{\gamma },N)}\sum\limits_{l\in
B}f_{l}\sum\limits_{k\in B}\gamma _{k}^{l}\xi _{n}^{k}=\frac{1}{w_{n}(%
\mathbf{\gamma },N)}\sum\limits_{l\in B}f_{l}\zeta _{nl},
\end{equation*}%
for all $n\in B$, and further define the column vector%
\begin{equation*}
\mathbf{v}=[f_{l}:l\in B]^{T},
\end{equation*}%
the matrix%
\begin{equation*}
\mathbf{Z}=[(\zeta _{nl}:n,l\in B)],
\end{equation*}%
and the diagonal matrix of weights $\mathbf{W}$, with diagonal elements $%
w_{l}(\mathbf{\gamma },N),$ $l\in B$. As a result, we can write the latter
display as%
\begin{equation*}
\mathbf{v}=\mathbf{W}^{-1}\mathbf{Zv\Rightarrow }\left( \mathbf{I}_{N}-%
\mathbf{W}^{-1}\mathbf{Z}\right) \mathbf{v}=\mathbf{0}_{N},
\end{equation*}%
for arbitrary $\mathbf{v}$, where $\mathbf{0}_{N}$ the $N\times 1$ zero
vector, $\mathbf{I}_{N}$ denotes the $N\times N$ identity matrix, and
therefore we must have $\mathbf{Z}=\mathbf{W}.$ Consequently, $\zeta
_{nl}=w_{n}(\mathbf{\gamma },N)\delta _{nl},$ for all $n,l\in B$, and the
orthogonality conditions (\ref{IDTTCond}) hold.\newline
$(\Leftarrow )$ Trivially, assuming that the orthogonality conditions (\ref%
{IDTTCond}) hold, we can write (\ref{InvDTT1}) as%
\begin{equation*}
\mathcal{IT}_{\xi _{n},B}^{g_{k}}(1)=\frac{1}{w_{n}(\mathbf{\gamma },N)}%
\sum\limits_{l\in B}f_{l}w_{n}(\mathbf{\gamma },N)\delta _{nl}=f_{n},
\end{equation*}%
so that the arbitrary $\mathbf{v}$ is recovered uniquely, as required.
\end{proof}

There exists at least one such sequence $\mathbf{\gamma }$ that satisfies
the latter theorem, i.e., the DFT and IDFT. To see this assume that $%
B_{N}=\{0,1,...,N-1\},$ and take the distinct $\gamma _{n}=e^{-i2\pi \frac{n%
}{N}},$ $n\in B_{N},$ and set $\xi _{n}=\overline{\gamma }_{n},$ such that
the conditions (\ref{IDTTCond}) reduce to%
\begin{equation*}
\sum\limits_{k\in B_{N}}\gamma _{k}^{l}\overline{\gamma }_{n}^{k}=\sum%
\limits_{k=0}^{N-1}e^{-i2\pi \frac{kl}{N}}e^{i2\pi \frac{nk}{N}%
}=\sum\limits_{k=0}^{N-1}e^{i2\pi \frac{k}{N}(l-n)}=N\delta _{nl},
\end{equation*}%
for all $n,l\in B_{N}$, and therefore, the IDFT\ is a special case of
Theorem \ref{DTTInversion}, with $w_{n}(\mathbf{\gamma },N)=N,$ for all $%
n\in B_{N}$.

When we work in a general field $\mathbb{F}$ (other than $\mathbb{C}),$ the
standard generalization of the DFT via the principal $N^{th}$ root of unity,
is also a special case of the DTT. Recall that the principal $N^{th}$ root
of unity over the field $\mathbb{F}$, is an element $r\in \mathbb{F}$ such
that%
\begin{equation*}
r^{N}=1,\text{ and }\sum\limits_{j=0}^{N-1}r^{jk}=0,
\end{equation*}%
for all $k=0,1,...,N-1.$ In this case, we have the driving sequence $\gamma
_{n}=r^{n},$ $n=0,1,...,N-1,$ and setting%
\begin{equation*}
\xi _{n}=\gamma _{n}^{-1}=r^{-n},
\end{equation*}%
we have that the conditions (\ref{IDTTCond}) reduce to%
\begin{equation*}
\sum\limits_{k=0}^{N-1}r^{kl}r^{-kn}=\sum\limits_{k=0}^{N-1}r^{k(l-n)}=N%
\delta _{nl},
\end{equation*}%
which\ is again, a special case of Theorem \ref{DTTInversion}, with $w_{n}(%
\mathbf{\gamma },N)=N,$ for all $n\in B_{N}$. The DFT is based on the
principal $N^{th}$ root of unity over $\mathbb{C}$ given by $r=e^{-i2\pi
\frac{1}{N}},$ with the roots being points on the unit circle of the complex
plane.

Since we would like to use a general sequence $\gamma _{k}$ in the DTT over $%
\mathbb{M}$, a natural question arises; is the DFT the only invertible DTT
using Theorem \ref{DTTInversion}? In other words, is there any driving
sequence $\mathbf{\gamma },$ other than $e^{-i2\pi \frac{1}{N}}$, for which
there exists a driving sequence $\mathbf{\xi }$ such that%
\begin{equation}
\sum\limits_{k=0}^{N-1}\gamma _{k}^{l}\xi _{n}^{k}=w_{n}(\mathbf{\gamma }%
,N)\delta _{nl},  \label{Conds1}
\end{equation}%
for all $n,l\in B_{N}$.

Let $\mathbf{\gamma }_{N}=[\gamma _{0},...,\gamma _{N-1}]^{T},$ $\mathbf{\xi
}_{N}=[\xi _{0},...,\xi _{N-1}]^{T},$ $\mathbf{\Xi }_{N}=\mathbb{V}(\xi
_{0},...,\xi _{N-1}),$ $\mathbf{\Gamma }_{N}^{-1}=[(\psi _{ij})],$ where%
\begin{equation}
\psi _{ij}=(-1)^{N-1-i}\frac{\sum\limits_{\substack{ 0\leq
j_{0}<...<j_{N-1-i}\leq N-1  \\ j_{0}<...<j_{N-1-i}\neq j}}\gamma
_{j_{0}}...\gamma _{j_{N-1-j}}}{\prod\limits_{m=0,m\neq j}^{N-1}(\gamma
_{j}-\gamma _{m})},  \label{Psiijs}
\end{equation}%
and $\mathbf{W}=diag(w_{0}(\mathbf{\gamma },N),$ $w_{1}(\mathbf{\gamma },N),$
$...,$ $w_{N-1}(\mathbf{\gamma },N)),$ the diagonal matrix of the weights.
In view of the DTT of Equation (\ref{MatRepDTT}), the Vandermonde matrix $%
\mathbf{\Gamma }_{N}=\mathbb{V}(\gamma _{0},...,\gamma _{N-1})$ plays a
pivotal role in the DTT and its inversion, since we can rewrite the
conditions (\ref{Conds1}) as the system of equations%
\begin{equation}
\mathbf{\Xi }_{N}\mathbf{\Gamma }_{N}=\mathbf{W}\Rightarrow \mathbf{\Xi }%
_{N}=\mathbf{W\Gamma }_{N}^{-1},  \label{CondsVandermonde}
\end{equation}%
so that $\mathbf{\xi }_{N}$ is the second column of $\mathbf{\Xi }_{N}$,
with the $q^{th}$ row of $\mathbf{\Xi }_{N}$ given by%
\begin{equation}
\left[ \xi _{q}^{0},\xi _{q}^{1},...,\xi _{q}^{N-2},\xi _{q}^{N-1}\right] =%
\left[ w_{q}(\mathbf{\gamma },N)\psi _{q,0},w_{q}(\mathbf{\gamma },N)\psi
_{q,1},...,w_{q}(\mathbf{\gamma },N)\psi _{q,N-1}\right] ,
\label{KsiSystemEqs}
\end{equation}%
and the first equation $1=w_{q}(\mathbf{\gamma },N)\psi _{q,0},$ giving
immediately the weight functions as%
\begin{equation}
w_{q}(\mathbf{\gamma },N)=\frac{1}{\psi _{q,0}},  \label{Weights1}
\end{equation}%
for all $q=0,1,2,...,N-1.$ Clearly, the system of equations (\ref%
{CondsVandermonde}) has a unique solution $\mathbf{\xi }_{N}$ provided that $%
\mathbf{\Xi }_{N}$ is a Vandermonde matrix based on this $\mathbf{\xi }_{N}$%
, and the latter depends on the form of the driving vector $\mathbf{\gamma }%
_{N}$. Then, uniqueness of $\mathbf{\xi }_{N}$\ follows from (\ref%
{CondsVandermonde}) and (\ref{KsiSystemEqs}) since $\mathbf{\Gamma }%
_{N}^{-1} $ exists and it is unique, even if ill-conditioned. However, the\
inverse of a square weighted Vandermonde matrix is not a Vandermonde matrix
in general, which, in this setup, is a requirement for the IDTT to be driven
by a specific vector $\mathbf{\xi }_{N}$, and therefore, existence of $%
\mathbf{\xi }_{N}$ needs to be further investigated. The following result
presents the general conditions required for $\mathbf{\xi }_{N}$ to exist
and be unique.

\begin{lemma}
\label{LemmaInv1}The system of equations (\ref{CondsVandermonde}) has a
unique solution $\mathbf{\xi }_{N}=[\xi _{0},...,\xi _{N-1}]^{T}$ if and
only if the weighted inverse Vandermonde matrix $\mathbf{W\Gamma }_{N}^{-1}$
is the Vandermonde matrix $\mathbf{\Xi }_{N}=\mathbb{V}(\xi _{0},...,\xi
_{N-1}),$ with $\mathbf{\xi }_{N}$ defined by%
\begin{equation}
\xi _{n}=\frac{\psi _{n,1}}{\psi _{n,0}},  \label{CondDTTMI1}
\end{equation}%
where $\psi _{n,0}$ and $\psi _{n,1}$ are given by Equation (\ref{Psiijs}),
the weights satisfy Equation (\ref{Weights1}), and the driving vector $%
\mathbf{\gamma }_{N}$ $=$ $[\gamma _{0},\gamma _{1},...,$ $\gamma
_{N-1}]^{T} $ is such that it satisfies%
\begin{equation}
\psi _{n,k}=\psi _{n,1}^{k}\psi _{n,0}^{1-k},  \label{CondDTTMI2}
\end{equation}%
for all $k=1,2,...,N-1,$ and $n=0,1,...,N-1,$ with $\mathbf{\Gamma }%
_{N}^{-1}=[(\psi _{ij})].$\newline
In particular, every ratio of two nodes is an $N$-th root of unity and the
nodes vector $\mathbf{\gamma }_{N}$ must satisfy%
\begin{equation}
\left( \frac{\gamma _{m}}{\gamma _{n}}\right) ^{N}=1.  \label{condgammas}
\end{equation}%
for all $m,n=0,1,2,...,N-1.$
\end{lemma}

\begin{proof}
$(\Rightarrow )$ Assume that the system (\ref{CondsVandermonde}) has a
unique solution $\mathbf{\xi }_{N}.$ In order for $\mathbf{\Xi }_{N}$ in (%
\ref{CondsVandermonde}) to be the Vandermonde matrix $\mathbf{\Xi }_{N}=%
\mathbb{V}(\xi _{0},...,\xi _{N-1}),$ we must have that equations (\ref%
{KsiSystemEqs}) hold, so that $\xi _{q}$ satisfies%
\begin{eqnarray*}
1 &=&w_{q}(\mathbf{\gamma },N)\psi _{q,0}\Longrightarrow w_{q}(\mathbf{%
\gamma },N)=\frac{1}{\psi _{q,0}}, \\
\xi _{q}^{1} &=&w_{q}(\mathbf{\gamma },N)\psi _{q,1}\Longrightarrow \xi _{q}=%
\frac{\psi _{q,1}}{\psi _{q,0}}, \\
&&... \\
\xi _{q}^{N-1} &=&w_{q}(\mathbf{\gamma },N)\psi _{q,N-1}\Longrightarrow \xi
_{q}=\left( \frac{\psi _{q,N-1}}{\psi _{q,0}}\right) ^{\frac{1}{N-1}},
\end{eqnarray*}%
for all $q=0,1,2,...,N-1.$ As a result, we have%
\begin{equation*}
\xi _{q}=\frac{\psi _{q,1}}{\psi _{q,0}}=\left( \frac{\psi _{q,2}}{\psi
_{q,0}}\right) ^{\frac{1}{2}}=...=\left( \frac{\psi _{q,N-1}}{\psi _{q,0}}%
\right) ^{\frac{1}{N-1}},
\end{equation*}%
with $w_{q}(\mathbf{\gamma },N)=1/\psi _{q,0},$ so that solving for $\psi
_{q,j}$ leads to%
\begin{equation*}
\psi _{q,j}=\frac{\psi _{q,1}^{k}}{\psi _{q,0}^{k-1}}=\psi _{q,1}^{k}\psi
_{q,0}^{1-k},
\end{equation*}%
for all $k=1,2,...,N-1,$ and $q=0,1,2,...,N-1,$ and Equations (\ref{Weights1}%
), (\ref{CondDTTMI1}) and (\ref{CondDTTMI2}) are satisfied, as entertained.%
\newline
\textbf{The form of} $\mathbf{\gamma }_{N}:$ Recall that the inverse
Vandermonde matrix can be described naturally using the Lagrange
interpolation polynomials, i.e., for $n=0,\ldots ,N-1$, define
\begin{equation*}
L_{n}(x)=\prod_{\substack{ m=0 \\ m\neq n}}^{N-1}\frac{x-\gamma _{m}}{\gamma
_{n}-\gamma _{m}},
\end{equation*}%
such that the polynomials satisfy
\begin{equation*}
L_{n}(\gamma _{j})=\delta _{nj}.
\end{equation*}%
Expanding $L_{n}(x)$ in terms of powers of $x$ as
\begin{equation*}
L_{n}(x)=\sum_{k=0}^{N-1}\psi _{n,k}x^{k},
\end{equation*}%
the $n$-th row of $\mathbf{\Gamma }_{N}^{-1}=[(\psi _{ij})]$ is precisely $%
[\psi _{n,0},\psi _{n,1},\ldots ,\psi _{n,N-1}].$ Using the standard
approach in the literature, let $e_{r}$ denote the $r$-th elementary
symmetric polynomial (see \cite{ballantine2025partitions}). Then the
coefficients of $L_{n}(x)$ are given by
\begin{equation*}
\psi _{n,k}=\frac{(-1)^{N-1-k}e_{N-1-k}\left( \{\gamma _{m}:m\neq n\}\right)
}{\prod_{m\neq n}(\gamma _{n}-\gamma _{m})}.
\end{equation*}%
In particular, the constant coefficient is
\begin{equation*}
\psi _{n,0}=\frac{\prod_{m\neq n}\gamma _{m}}{\prod_{m\neq n}(\gamma
_{n}-\gamma _{m})},
\end{equation*}%
while the coefficient of $x$ is
\begin{equation*}
\psi _{n,1}=-\frac{\sum_{m\neq n}\prod_{\substack{ \ell \neq n,m}}\gamma
_{\ell }}{\prod_{m\neq n}(\gamma _{n}-\gamma _{m})}.
\end{equation*}%
Finally, because none of the $\gamma _{m}$ is zero, we have $\psi _{n,0}\neq
0$, so that using equation (\ref{CondDTTMI1}) we obtain
\begin{equation*}
\xi _{n}=-\frac{\sum_{m\neq n}\prod_{\substack{ \ell \neq n,m}}\gamma _{\ell
}}{\prod_{m\neq n}\gamma _{m}},
\end{equation*}%
and dividing each term in the numerator by $\prod_{m\neq n}\gamma _{m}$
gives the general form of the $\xi $s as a function of the $\gamma $s as%
\begin{equation}
\xi _{n}=-\sum_{\substack{ m=0 \\ m\neq n}}^{N-1}\frac{1}{\gamma _{m}},
\label{ksi_n}
\end{equation}%
for all $n=0,\ldots ,N-1$. Thus, the ratio of the second entry to the first
entry in the $n$-th row of $\mathbf{\Gamma }_{N}^{-1}$ is simply the
negative sum of the reciprocals of all the Vandermonde nodes other than $%
\gamma _{n}$.\newline
Consider first $m\neq n$, and using Equation (\ref{CondDTTMI1}), we can
rewrite the assumed relation as
\begin{equation*}
\psi _{n,k}=\psi _{n,0}\xi _{n}^{k},
\end{equation*}%
for all $k=0,1,2,...,N-1,$ and $n=0,1,...,N-1,$ and therefore%
\begin{equation*}
L_{n}(x)=\sum_{k=0}^{N-1}\psi _{n,k}x^{k}=\sum_{k=0}^{N-1}\psi _{n,0}\xi
_{n}^{k}x^{k}=\psi _{n,0}\sum_{k=0}^{N-1}\xi _{n}^{k}x^{k}=\psi
_{n,0}\sum_{k=0}^{N-1}(\xi _{n}x)^{k}.
\end{equation*}%
Since
\begin{equation*}
L_{n}(\gamma _{m})=0,
\end{equation*}%
we obtain
\begin{equation*}
0=\psi _{n,0}\sum_{k=0}^{N-1}(\xi _{n}\gamma _{m})^{k},
\end{equation*}%
and since $\psi _{n,0}\neq 0$, it follows that
\begin{equation*}
\sum_{k=0}^{N-1}(\xi _{n}\gamma _{m})^{k}=0.
\end{equation*}%
Consequently, we can write%
\begin{equation*}
\sum_{k=0}^{N-1}(\xi _{n}\gamma _{m})^{k}=\frac{(\xi _{n}\gamma _{m})^{N}-1}{%
(\xi _{n}\gamma _{m})-1}=0\Rightarrow (\xi _{n}\gamma _{m})^{N}=1,
\end{equation*}%
for all $m\neq n,$ and
\begin{equation*}
\xi _{n}\gamma _{m}\neq 1.
\end{equation*}%
Thus, for fixed $n$, the $N-1$ numbers
\begin{equation*}
\left\{ \xi _{n}\gamma _{m}:m\neq n\right\}
\end{equation*}%
are $N-1$ distinct nontrivial $N$-th roots of unity. There is exactly one $N$%
-th root of unity missing from this set. Because the nodes $\gamma _{m}$ are
distinct, the remaining number
\begin{equation*}
\xi _{n}\gamma _{n}
\end{equation*}%
cannot coincide with any of the other $N-1$ values. Hence it must be the
missing root of unity, namely $1$. Therefore
\begin{equation*}
\xi _{n}\gamma _{n}=1,
\end{equation*}%
and it follows that
\begin{equation*}
\xi _{n}=\frac{1}{\gamma _{n}}.
\end{equation*}%
Moreover, for every $m\neq n$,
\begin{equation*}
\left( \frac{\gamma _{m}}{\gamma _{n}}\right) ^{N}=(\xi _{n}\gamma
_{m})^{N}=1,
\end{equation*}%
and hence (\ref{condgammas}) is established.\newline
$(\Leftarrow )$ Now assume that Equations (\ref{Weights1}), (\ref{CondDTTMI1}%
) and (\ref{CondDTTMI2}) hold. Then trivially, since
\begin{equation*}
\left( \frac{\psi _{n,k}}{\psi _{n,0}}\right) ^{\frac{1}{k}}=\left( \frac{%
\psi _{n,1}^{k}\psi _{n,0}^{1-k}}{\psi _{n,0}}\right) ^{\frac{1}{k}}=\frac{%
\psi _{n,1}}{\psi _{n,0}}=\xi _{n},
\end{equation*}%
for all $k=1,2,...,N-1,$ and $n=0,1,2,...,N-1,$ it follows that $\mathbf{\Xi
}_{N}=\mathbb{V}(\xi _{0},...,\xi _{N-1})=\mathbf{W\Gamma }_{N}^{-1}$ and (%
\ref{CondsVandermonde}) holds$.$ Furthermore, the $\xi $s are obtained as a
function of the $\gamma $s using Equation (\ref{ksi_n}).
\end{proof}

Note that the driving vector $\mathbf{\gamma }_{N}$ $=$ $[\gamma _{0},\gamma
_{1},...,$ $\gamma _{N-1}]^{T}$ is obtained by conditions (\ref{CondDTTMI2})
which give $N$-equations with $N$-unknowns. Unfortunately, the root-of-unity
condition is highly restrictive. To obtain an inversion method applicable to
more general driving vectors $\mathbf{\gamma }_{N}$, we now introduce a
matrix-based inversion procedure.

\subsection{DTT Matrix Inversion}

In what follows we investigate the inversion of the DTT specifically for
sets of the form $B_{N}$ and we begin with a general DTT matrix inversion in
the following definition, that will allow us to apply the DTT and invert it
for general $\mathbf{\gamma }_{N}.$

\begin{definition}[DTT Matrix Inversion]
\label{DTTMI}Assume that $B=\{0,1,...,N-1\}$ $\in $ $\mathcal{B}(\mathbb{N}%
), $ $N>1$, and consider two vectors $\mathbf{f}%
_{N}=[f_{0},f_{1},...,f_{N-1}]^{T},$ and $\mathbf{\gamma }_{N}$ $=$ $[\gamma
_{0},\gamma _{1},...,\gamma _{N-1}]^{T},$ $\gamma _{n}\neq 0,$ $\forall n\in
B,$ $\gamma _{n}$ distinct, and $\mathbf{f}_{N}\neq \mathbf{0}.$ The DTT at $%
t=1$, of $\mathbf{f}_{N}$ driven by $\mathbf{\gamma }_{N}$ is given by the
vector $\mathbf{g}_{N}=$ $[g_{0},g_{1},...,g_{N-1}]^{T}$ with%
\begin{equation*}
\mathbf{g}_{N}=\mathbf{\Gamma }_{N}\mathbf{f}_{N},
\end{equation*}%
where $\mathbf{\Gamma }_{N}=\mathbb{V}(\gamma _{0},...,\gamma _{N-1})$ the
Vandermonde matrix of Equation (\ref{Gammamat}), with%
\begin{equation*}
g_{k}=\sum\limits_{n=0}^{N-1}f_{n}\gamma _{k}^{n},
\end{equation*}%
$k=0,1,...,N-1.$ Then the DTT matrix inverse (DTTMI) of $\mathbf{g}_{N}$
driven by the matrix $\mathbf{\Psi }=[(\psi _{ij})]$ is defined by%
\begin{equation*}
\mathbf{f}_{N}^{-}:=\mathbf{\Psi g}_{N},
\end{equation*}%
where $\mathbf{f}_{N}^{-}=[f_{0}^{-},f_{1}^{-},...,f_{N-1}^{-}]^{T},$ with%
\begin{equation*}
f_{n}^{-}:=\sum\limits_{k=0}^{N-1}g_{k}\psi _{nk},
\end{equation*}%
and it recovers the original vector perfectly provided that $\mathbf{\Psi }=%
\mathbf{\Gamma }_{N}^{-1}$, i.e.,%
\begin{equation*}
\mathbf{f}_{N}^{-}=\mathbf{\Psi g}_{N}=\mathbf{\Gamma }_{N}^{-1}\mathbf{%
\Gamma }_{N}\mathbf{f}_{N}=\mathbf{f}_{N},
\end{equation*}%
for all $n=0,1,...,N-1.$
\end{definition}

The DTTMI is straightforward and can be utilized to invert the DTT based on
general driving vectors $\mathbf{\gamma }_{N}$ that satisfy mild conditions,
such as the being non-zero and distinct, however it does not necessarily
correspond to an IDTT. Moreover, the ill--conditioning of the Vandermonde
matrix is responsible for underflows and overflows that can cause
irreparable exceptions in machine arithmetic when we try to obtain the
inverse. Therefore, out of necessity, we will offer shortly an alternative
way of calculating $\mathbf{\xi }_{N}$ or $\mathbf{\Gamma }_{N}^{-1}$ via
Monte Carlo. Finally note that the DTT inversions of Theorem \ref%
{DTTInversion} and Definition \ref{DTTMI} do not require the driving
sequence $\mathbf{\gamma }$\ to be orthogonal (as in the DFT case), but they
work for any sequence under mild conditions.

\begin{example}[DTTMI for small $N$]
\label{IDTTex}We will consider a specific sequence $\mathbf{f}=(f_{0},$ $%
f_{1},$ $...),$ and driving sequences $\mathbf{\gamma }$ $=$ $(\gamma
_{0},\gamma _{1},...)$, with distinct $\gamma _{n}\neq 0,$ $\forall n\in
B_{N},$ where $B_{N}=$ $\{$ $0,1,$ $...,N-1\},$ with $N<+\infty ,$ so we do
not have to worry about convergence issues with the DTT and DTTMI.\newline
Take $N=4$ with $\mathbf{f}=(1+0i,2-1i,0-1i,-1+2i,...),$ or $\mathbf{f}%
_{N}=[1+0i,2-1i,0-1i,-1+2i]^{T},$ and let $\gamma _{k}=\frac{1}{k+1},$ $k\in
B_{4}=\{0,1,2,3\}$, be the driving vector. Then the DTT vector is given by $%
\mathbf{g}_{N}=[2+0i,$ $1.875-0.5i,$ $1.62963-0.37037i,$ $%
1.484375-0.28125i]^{T},$ and the DTTMI recovers $\mathbf{f}$ exactly, based
on $\mathbf{\Psi }=\mathbb{V}^{-1}(1,0.5,0.3333,0.25).$\newline
Now consider as driving vector the shifted Catalan numbers $\gamma _{k}=%
\frac{(2(k+1))!}{(k+1)!(k+1)!},$ $k\in B_{4},$ so that the DTT becomes $%
\mathbf{g}_{N}=[2+0i,$ $-3+10i,$ $-114+220i,$ $-2715+5278i]^{T}.$ The DTTMI,
once again, recovers the original sequence perfectly, based on $\mathbf{\Psi
}=\mathbb{V}^{-1}(1,2,5,14).$
\end{example}

For excellent reviews on classic and the latest methods of finding the
inverse of the Vandermonde matrix see \cite{hosseini2019closed} and \cite%
{beinert2023phase}. Unfortunately, even for moderate or large $N$, e.g., $%
N>50$, these classic and recent methods are not able to recover the
Vandermonde matrix inverse for any distinct and non-zero $\mathbf{\gamma }%
_{N}$\ (e.g., machine overflows), so instead we propose and utilize a Monte
Carlo simulation approach, which we collect next.

\begin{remark}[Vandermonde matrix inversion]
\label{VandermondeInversion}We briefly discuss the background and steps
required to accomplish the inversion.

\begin{enumerate}
\item Multivariate Gaussian distribution: Recall the multivariate Gaussian
distribution (see \cite{muirhead2009aspects}), which is used to describe
random vectors. More precisely, the Gaussian distribution of a $p$%
-dimensional random vector $\mathbf{X}$ is denoted by $\mathcal{N}_{p}(%
\mathbf{\mu },\mathbf{\Sigma }),$ with density%
\begin{equation*}
f(\mathbf{x})=\left\vert 2\pi \mathbf{\Sigma }\right\vert ^{-\frac{1}{2}%
}\exp \left\{ -\frac{1}{2}(\mathbf{x}-\mathbf{\mu })^{T}\mathbf{\Sigma }%
^{-1}(\mathbf{x}-\mathbf{\mu })\right\} ,
\end{equation*}%
$\mathbf{x}\in \Re ^{p}$, where $\mathbf{\mu }$ denotes the mean vector and $%
\mathbf{\Sigma }$ is a symmetric, positive semi-definite variance-covariance
matrix, i.e., $\mathbb{E}(\mathbf{X})=\mathbf{\mu }$, and $V(\mathbf{X})=%
\mathbb{E}(\mathbf{X}-\mathbf{\mu })(\mathbf{X}-\mathbf{\mu })^{T}=\mathbf{%
\Sigma }$, where $\mathbb{E}$ denotes expectation.

\item Monte Carlo simulation: The Strong Law of Large Numbers (SLLN, see
\cite{Billingsley2013}) allows us to approximate the means of functions of
random objects (variables, vectors, matrices, sets etc). For Gaussian random
vectors in particular, let $\mathbf{X}_{1},\mathbf{X}_{2},...,\mathbf{X}%
_{L}, $ denote a large random sample $($e.g., $L\geq 100000)$ from the
Gaussian distribution $\mathcal{N}_{p}(\mathbf{\mu },\mathbf{\Sigma })$.
Then the sample mean satisfies%
\begin{equation}
\overline{\mathbf{X}}:=\frac{1}{L}\sum\limits_{l=1}^{L}\mathbf{X}_{l}\overset%
{a.s.}{\rightarrow }E(\mathbf{X})=\mathbf{\mu },  \label{Normmean}
\end{equation}%
and sample variance-covariance matrix satisfies%
\begin{equation}
\mathbf{S}:=\frac{1}{L-1}\sum\limits_{l=1}^{L}\left( \mathbf{X}_{l}-%
\overline{\mathbf{X}}\right) \left( \mathbf{X}_{l}-\overline{\mathbf{X}}%
\right) ^{T}\overset{a.s.}{\rightarrow }V(\mathbf{X})=\mathbf{\Sigma },
\label{SampleVar}
\end{equation}%
such that%
\begin{equation}
\mathbf{S}^{-1}\overset{a.s.}{\rightarrow }\mathbf{\Sigma }^{-1},
\label{InvSampleVar}
\end{equation}%
and the larger $L$ becomes the better the approximation. Note that even if $%
\mathbf{\Sigma }$ is singular, we can still generate random vectors from the
Gaussian distribution, such that (\ref{InvSampleVar}) still holds with
matrix inverse replaced by generalized inverse, say $\mathbf{S}^{+}$ and $%
\mathbf{\Sigma }^{+}$, the Moore-Penrose generalized inverse (MPGI) of $%
\mathbf{S}$ and $\mathbf{\Sigma },$ respectively$.$ When $rank(\mathbf{%
\Sigma })=r<p,$ we typically use the unbiased estimator of $\mathbf{\Sigma }%
^{+}$ given by $\frac{p-r-2}{p}\mathbf{S}^{+}.$

\item Vandermonde inverse approximation: Now take $\mathbf{\mu }=\mathbf{0},$
and $\mathbf{\Sigma }=\mathbf{\Gamma }_{N}\mathbf{\Gamma }_{N}^{T}>0,$ such
that $\mathbf{\Sigma }$ is a symmetric and positive definite matrix,
however, in machine arithmetic, $\mathbf{\Sigma }$ may be reported in the
computer as singular as $N$ increases, owing to the ill-conditioning of the
Vandermonde matrix $\mathbf{\Gamma }_{N}$. Using Equation (\ref{InvSampleVar}%
) we can write%
\begin{equation*}
\mathbf{S}^{-1}\overset{a.s.}{\rightarrow }\left( \mathbf{\Gamma }_{N}%
\mathbf{\Gamma }_{N}^{T}\right) ^{-1}=\left( \mathbf{\Gamma }_{N}^{T}\right)
^{-1}\mathbf{\Gamma }_{N}^{-1},
\end{equation*}%
and therefore upon multiplying from the right with $\mathbf{\Gamma }_{N}^{T}$%
, we obtain an approximation of the Vandermonde matrix inverse via%
\begin{equation}
\mathbf{\Gamma }_{N}^{T}\mathbf{S}^{-1}\overset{a.s.}{\rightarrow }\mathbf{%
\Gamma }_{N}^{-1},  \label{VandermondeMC}
\end{equation}%
which requires a single matrix inversion of a well behaved matrix $\mathbf{S}
$. For justification on the validity of the operations above involving
almost sure convergence, see \cite{Ferguson}.
\end{enumerate}
\end{remark}

Whereas the matrix $\mathbf{\Gamma }_{N}$ is a square Vandermonde matrix by
construction of the IDTT or DTTMI, there are applications where we require
inversion of a general $n\times p$ $(p\leq n)$ Vandermonde matrix $\mathbf{X}
$ based on the vector $\mathbf{x}=[x_{1},...,x_{n}]^{T}$. Assuming that the $%
x$s are distinct, the matrix $\mathbf{X}$ is of full rank, i.e., the rank is
equal to $\min (n,p)$. When $\min (n,p)=p,$ we have that $\mathbf{X}$\ has
full column rank and $\left( \mathbf{X}^{T}\mathbf{X}\right) ^{-1}$ exists,
whereas, if $\min (n,p)=n,$ we have that $\mathbf{X}$\ has full row rank and
$\left( \mathbf{XX}^{T}\right) ^{-1}$ exists. The construction above is
still viable as follows; in the first case, set $\mathbf{\Sigma }=\mathbf{X}%
^{T}\mathbf{X}>0,$ a positive definite, $p\times p$ symmetric matrix, with (%
\ref{InvSampleVar}) leading to $\mathbf{S}^{-1}\overset{a.s.}{\rightarrow }%
\left( \mathbf{X}^{T}\mathbf{X}\right) ^{-1},$ so that a left inverse of a
non-square matrix $\mathbf{X}$ is given by the $p\times n$ matrix%
\begin{equation*}
\mathbf{S}^{-1}\mathbf{X}^{T}\overset{a.s.}{\rightarrow }\mathbf{X}%
_{L}^{-1}=\left( \mathbf{X}^{T}\mathbf{X}\right) ^{-1}\mathbf{X}^{T},
\end{equation*}%
with $\mathbf{X}_{L}^{-1}\mathbf{X}=\mathbf{I}_{p},$ since $\mathbf{X}$ has
full column rank. Similarly, when $\mathbf{X}$ has full row rank, we take $%
\mathbf{\Sigma }=\mathbf{XX}^{T}>0,$ a positive definite, $n\times n$
symmetric matrix, with $\mathbf{S}^{-1}\overset{a.s.}{\rightarrow }\left(
\mathbf{XX}^{T}\right) ^{-1},$ so that a right inverse of $\mathbf{X}$ is
given by the $p\times n$ matrix%
\begin{equation*}
\mathbf{X}^{T}\mathbf{S}^{-1}\overset{a.s.}{\rightarrow }\mathbf{X}_{R}^{-1}=%
\mathbf{X}^{T}\left( \mathbf{XX}^{T}\right) ^{-1},
\end{equation*}%
with $\mathbf{XX}_{R}^{-1}=\mathbf{I}_{n}.$ We illustrate the approximation
to a square Vandermonde inverse in the following.

\begin{example}[Vandermonde inverse for large $N$]
\label{VanderInvex}In Table \ref{PolysTable}, we illustrate the Monte Carlo
simulations for a square Vandermonde matrix inversion based on $L=1000000$
realizations from a $\mathcal{N}_{p}(\mathbf{0},\mathbf{\Sigma }),$
distribution with $\mathbf{\Sigma }=\mathbf{\Gamma }_{N}\mathbf{\Gamma }%
_{N}^{T}$. We applied the approximation for $N=5,$ $10,$ $50,$ $100,$ $200,$
and $500$, using the driving sequence $\gamma _{k}=\frac{1}{k+2},$ $k\in
B_{N}=\{0,1,2,...,N-1\}$, in all cases. In order to verify the performance
of the Monte Carlo simulation, we present the Frobenius norm $d_{1}$ between
$\mathbf{S}$ (the estimator of the variance-covariance matrix), and the true
matrix $\mathbf{\Gamma }_{N}\mathbf{\Gamma }_{N}^{T}$, based on Equation (%
\ref{SampleVar}). We also present the Frobenius norm $d_{2}$ between $%
\mathbf{\Gamma }_{N}$ and its estimate based on Equation (\ref{VandermondeMC}%
) given by $\mathbf{\Gamma }_{N}\mathbf{\Gamma }_{N}^{+}\mathbf{\Gamma }%
_{N}, $ with $\mathbf{\Gamma }_{N}^{+}=\mathbf{\Gamma }_{N}^{T}\mathbf{S}%
^{+},$ where we used the MPGI $\mathbf{S}^{+}$ to avoid any problems with
computer calculations. Clearly, even as $N$ increases the approximation
performs well, however, when we take the inverse $\mathbf{S}^{+}$ we still
encountered issues with machine arithmetic. One of the major advantages of
Monte Carlo simulation is that increasing $L$ alleviates these issues,
albeit at great computational cost and execution time, which currently makes
this approach impractical for $N>1000$.
\end{example}

\begin{table}[tbp]
\centering{\
\begin{tabularx}{\textwidth}{|X|X|X|X|X|}
\hline
$N$ & $d_1$& $d_2$ & Run Time  \\ \hline
5 & 0.009220962 & 0.007652776 & 0.75 seconds \\\hline
10 & 0.008592731& 0.009728198 & 13.87 seconds \\\hline
50 & 0.08639514& 0.02522218 & 343.11 seconds \\\hline
100 & 0.09686994 & 0.02779924 & 1636.69 seconds \\\hline
200 & 0.3155341& 0.03177035 & 5151.87 seconds \\\hline
500 & 0.1666946& 0.05948064 & 152938.30 seconds \\\hline
\end{tabularx}
}
\caption{Monte Carlo simulations for an $N\times N$ Vandermonde matrix
inversion. See the text for details and discussion.}
\label{PolysTable}
\end{table}

\section{Applications}

\label{Applications}In what follows, we present general applications of the
DTT and its inverse to the mathematical sciences.

\subsection{DTT based encoding example}

Cryptography refers to the practice and study of techniques for secure
communication between a sender (Alice) and a receiver (Bob), while an
adversary is eavesdropping (Eve). For reviews on classic methods, as well as
recent developments see \cite{biggs2008codes}, \cite{zheng2022modern}, \cite%
{salami2023cryptographic}, \cite{imran2024quantum}, \cite%
{sasikumar2024comprehensive}, and the references therein. We will propose a
general method based on the DTT, that does not use standard approaches like
encryption and decryption based on public or private keys (although standard
methods can still be applied for extra security). The closest existing
method to the proposed obfuscation method, is the DFT (see \cite%
{massey1998discrete}, \cite{vaudenay1999security}, \cite{chan2004block},
\cite{roche2008provable}, \cite{salami2023cryptographic}, and the
Reed-Solomon codes \cite{reed1960polynomial}), and it served as motivation
for this example, since the DFT is a special case of the DTT.

In order to avoid computer overflows, it is recommended that the reader uses
uniformly bounded sequences; more precisely, choosing a sequence such as $%
\gamma _{k}=(k+1)^{2}$ or $e^{k},$ $k=0,1,2,...,N-1$, will rapidly lead to
an overflow as $N$ increases, even in state of art machines.

Now assume that Alice wishes to send Bob the message Msg="Math rules!", with
$N=11$, while Eve tries to intercept and decode it. We assume that Alice and
Bob have hidden knowledge of the list of symbols that may be transmitted,
with Alice and Bob aware of all the sequences $\mathbf{\gamma }_{1},...,%
\mathbf{\gamma }_{M},$ that can be used to obtain the DTT of the message
Msg. Note here that each $\mathbf{\gamma }_{i},$ $i=1,2,...,M$, for any $N>0$
can be obtained offline and disseminated covertly to Alice and Bob, thus
making the time from coding a message, transmitting it and decoding it,
fairly constant.

The index $I$ is the first number transmitted, and it indicates the sequence
$\mathbf{\gamma }_{I}$ that was used in the DTT to code the message before
transmission. Even if Eve retrieves this number, she doesn't know which
sequence this corresponds to, since she has no access to the $\mathbf{\gamma
}$ sequences. As a result, only Bob will know the sequence $\mathbf{\gamma }%
_{I}$ to be used in the DTTMI to decode the message, for $I=1,2,...,M,$ with
larger $M$ increasing the complexity of the process, and making it harder to
decode, since different messages, or even the same message, can be sent
using different $I,$ eliminating any patterns that Eve might be able
identify. Once the initial number is transmitted by Alice (say as a real
number of 8 bytes, i.e., a double of 64bits), it is followed by the DTT
message itself, which consists of a collection of $N$ real numbers (for
simplicity, as $8N$ bytes).

Note here that even if a message is large $(N>100),$ it can be decomposed
into 10 characters at a time, which in most cases of $\mathbf{\gamma }$ will
typically lead to a Vandermonde matrix that is not ill-conditioned; this can
be checked before disseminating the sequences to Alice and Bob, to make sure
that there are no issues with inversion of the corresponding Vandermonde
matrix based on the sequence $\mathbf{\gamma }_{I}$. We will adapt this
approach in the cryptography example that follows. This is not unusual,
since many of the existing algorithms work on encryption/decryption of a
message in blocks, e.g., 64bit blocks (see \cite{salami2023cryptographic}).
When a message length is not a multiple of a chosen block length $l$,
padding can be used, i.e., augment the original message with trailing
characters to fill in the last block to have a fixed length of $l$
characters. Padding in cryptography involves the practice of adding extra
characters to a message before encryption to ensure that its length is a
multiple of the block size $l$. This helps prevent certain types of attacks
and ensures proper encryption and decryption processes in existing methods.
Alternatively, all sequences $\mathbf{\gamma }$ for $N=1,2,..,l\leq 10,$ can
be obtained offline and given covertly to Alice and Bob.

Consider building a collection of symbols and this can be accomplished in
any way desired; the 67 symbols we used were: lower case letters a, b, c, d,
e, f, g, h, i, j, k, l, m, n, o, p, q, r, s, t, u, v, w, x, y, and z; upper
case letters A, B, C, D, E, F, G, H, I, J, K, L, M, N, O, P, Q, R, S, T, U,
V, W, X, Y, Z; digits 0, 1, 2, 3, 4, 5, 6, 7, 8, 9; and special characters
such as the space character " ", comma character ",", period character ".",
the exclamation point "!", and the "at" character "@". We will keep things
simple by assuming that each symbol corresponds to a specific number, but
instead of the integers $1,2,...,67$, the number for each symbol is
uniformly sampled over some set, say $[-100,100]$ (see Table \ref%
{SymbolsTable}). Both Alice and Bob are aware of this hidden information.
When padding is required, it is applied using the @ character. Moreover, the
symbols chosen to be transmitted need not be letters in order to increase
the complexity of the message (e.g., a letter or phrase corresponds to an
unknown symbol such as a$\Longleftrightarrow \bigstar ,$ b$%
\Longleftrightarrow \square ,$ c$\Longleftrightarrow \blacksquare ,$ and so
forth), but we will keep things simple in this example.

\begin{table}[tbp]
\centering{\small
\begin{tabularx}{\textwidth}{|X|X|X|X|X|X|}
\hline
\multicolumn{6}{|c|}{Lower Case Letters}
\\\hline
a=74.09439 &b=4.841049 &c=-92.52354 &d=-69.50676 &e=-67.0486 &f=-47.57668
\\\hline
g=-51.75857 &h=68.35202 &i=-86.18408 &j=-38.46459 &k=-86.29489 &l=-64.16503
\\\hline
m=-84.3902 &n=54.84537 &o=-93.9089 &p=-63.77618 &q=-7.607419 &r=-70.3625
\\\hline
s=22.14637 &t=-1.579375 &u=-64.83118 &v=81.34243 &w=51.28614 &x=-83.69335
\\\hline
y=12.43306 &z=52.06726
\\\hline
\multicolumn{6}{|c|}{Upper Case Letters}
\\\hline
A=13.10272 &B=49.31613 &C=-14.54347 &D=31.76451 &E=86.46889 &F=44.7153
\\\hline
G=-15.96975 &H=-56.64244 &I=39.58658 &J=48.05374 &K=88.08317 &L=51.8521
\\\hline
M=13.4099 &N=-77.66365 &O=4.594217 &P=-51.61168 &Q=6.560883 &R=-94.25149
\\\hline
S=65.81979 &T=66.57866 &U=-71.87194 &V=-42.70737 &W=-34.79449 &X=-74.97178
\\\hline
Y=90.84026 &Z=-24.48978
\\\hline
\multicolumn{6}{|c|}{Digits}
\\\hline
0=14.95815 &1=-82.52232 &2=30.03144 &3=4.244889 &4=-34.88058 &5=16.18243
\\\hline
6=-84.52458 &7=81.14887 &8=43.87123 &9=-80.60847
\\\hline
\multicolumn{6}{|c|}{Other characters}
\\\hline
" " =31.55889 &,=38.17017 &.=-59.30511 &!=18.34118 &@=-57.0745&
 \\\hline
\end{tabularx}
}
\caption{Symbols used in the cryptography example. Each symbol corresponds
to a real number, uniformly sampled over $[-100,100]$. Alice and Bob are
given this table covertly.}
\label{SymbolsTable}
\end{table}

In Table \ref{DTTTable} we present the DTT analysis for the message "Math
rules!", which translates into the vector $\mathbf{f}_{N}=(13.409899,$ $%
74.094391,$ $-1.579375,$ $68.352017,$ $31.558894,$ $-70.362501,$ $%
-64.831183, $ $-64.165032,$ $-67.048597,$ $22.146366,$ $18.341185)$. The
sequences $\mathbf{\gamma }_{i}$, $i=1,2,...,$ $M=5$, along with the DTT
values (vector $\mathbf{g}_{N}$) are also presented. The length of the
message is $N=11$, and for these sequences it leads to well behaved
Vandermonde matrices. The Euclidean norm between the original sequence $%
\mathbf{f}$ and the DTTMI are given by 0.03289056, 1.364248e-05,
5.105436e-11, 2.178117e-11, and 1.890738e-07, so that the original message
is recovered by Bob exactly.

\begin{table}[tbp]
\centering{\small
\begin{tabularx}{\textwidth}{|X|X|X|}
\hline
Sequence& $\mathbf{\gamma }$ &DTT
\\\hline
$1/(k+1),$ \newline$ k=0,1,...,N-1$&
1, 0.5, 0.3333, 0.25,
 0.2, 0.1667, 0.1428, 0.125, 0.1111, 0.1, 0.0909&
-40.08393, 56.66484, 40.4371, 32.93668, 28.73527, 26.04518, 24.17017, 22.7858, 21.72035, 20.87428, 20.18573
\\\hline
$cos(k+2)^2,$\newline$k=0,1,...,N-1$&
0.1731, 0.98, 0.4272, 0.0804, 0.9219, 0.5683, 0.0211, 0.8301, 0.704, 0.000019, 0.712
&26.56439, -22.4514, 49.53783, 19.3983, 17.54337, 62.73893, 14.97844, 53.40165, 68.27403, 13.41135, 68.09797
\\\hline
$cos(k+1),$\newline$k=0,1,...,N-1$&
0.5403, -0.4161, -0.9899, \newline-0.6536,  0.2836,  0.9601, 0.7539, -0.1455, $-0.9111$, -0.839,  0.0044
&60.37636, -21.06286, \newline-97.32555, -44.86914, 35.89, \newline-6.933822, 65.76713, 2.403375, -79.02, -67.046, 13.73779
\\\hline
$sin(k+1),$\newline$k=0,1,...,N-1$&
0.84147,  0.9092,  0.1411, -0.7568, -0.9589, -0.2794,  0.65698,  0.9893, 0.41211, -0.544, $-0.9999$
&50.41492, 24.24, 24.0347, \newline-56.32167, -89.25698, -8.6204, 67.9066, -30.3882, 48.04271, \newline-33.59541, -100.2115
\\\hline
$Unif(-1,1)$&-0.7299,  0.65, 0.5319, -0.3601,  0.6871,  0.02083, 0.176,  0.4292,  0.3353, 0.4078,  0.8938
&-53.19813, 67.67917, 59.62883, -15.82339, 68.39966, 14.95329, 26.79606, 49.7359, 40.62669, 47.62556, 31.57885
 \\\hline
\end{tabularx}
}
\caption{The sequences $\mathbf{\protect\gamma }_{i}$, $i=1,2,...,M=5$,
along with the DTT values for the message "Math rules!" given by the vector $%
\mathbf{f}=(13.409899,$ $74.094391,$ $-1.579375, $ $68.352017,$ $31.558894,$
$-70.362501,$ $-64.831183,$ $-64.165032,$ $-67.048597,$ $22.146366,$ $%
18.341185)$. The length of the message is $N=11$. The Euclidean norm between
the original sequence $\mathbf{f}$ and the DTTMI are given by 0.03289056,
1.364248e-05, 5.105436e-11, 2.178117e-11, and 1.890738e-07.}
\label{DTTTable}
\end{table}

Using a different assignment of real numbers for the symbols of Table \ref%
{SymbolsTable}, and only one hidden sequence $\mathbf{\gamma }$ (none of the
ones presented herein), Eve intercepts the sequence of real numbers (DTT)
given by%
\begin{eqnarray*}
\mathbf{g}_{N} &=&(-1917.959,13245.93,-1781.262,-1277.915,-2136.627,-1401.33,
\\
&&-1982.884,16798.06,-2162.391,-1463.74,6958.445,43926.3, \\
&&1928.851,940.4619,-442.7533,630.6479,2739.06,51192.96, \\
&&4097.392,601.4443,-1258.352,-1728.008,-792.2136).
\end{eqnarray*}%
The cryptography community is invited to recover the true phrase that yields
this DTT, noting that we used blocking, there is no padding, and there is no
initial index $I$ in the transmitted message that Eve intercepts.

\subsection{DTT Computer Graphics}

We present an application of the DTT to image processing and computer
graphics (see \cite{gonzalez2009digital}, \cite{umbaugh2010digital}, and the
references therein), using Leonardo da Vinci's masterpiece portrait, the
Mona Lisa. The original image is converted to gray scale and is originally a
matrix of size $182\times 276=50232$ pixels. Alternatively, one can use the
Red, Green, Blue and Alpha channels separately, apply a TI or DTT filter,
and then recreate the processed image by combining all channels. Further
note that in the gray scale image (single channel) each pixel corresponds to
a value in the interval $[0,1]$, and any operations on a pixel are re-scaled
to values in $[0,1]$.

The original gray scale image is vectorized and converted to a $50232\times
1 $ vector $\mathbf{f}$, i.e., the vectorized image is created by stacking
each column of the original matrix in gray scale. Note that the image is
treated as a matrix with the top-left pixel being the $(1,1)$-element and
the bottom-right the $(182,276)$-element of the matrix. We used the DTT
approach in 5024 blocks, with block sizes of 10, except for the last block
which is of size 2.

In Figure \ref{MonaLisa1} we present from left to right, the original image
in color scale, the original image in gray scale, the DTT transform using
the sequence $\gamma _{k}=1/(k+1)$, $k=0,1,2,...,9$, and the recovered image
via the DTTMI. Note that the recovery from the DTT image is perfect, so that
in the context of the previous example in cryptography, the vector $\mathbf{f%
}$ (original image) can be transmitted by Alice, the DTT (blurred image) can
be intercepted by Eve, but only Bob can recover the original image.

\begin{figure}[t]
\centering \includegraphics[width=0.995\textwidth]{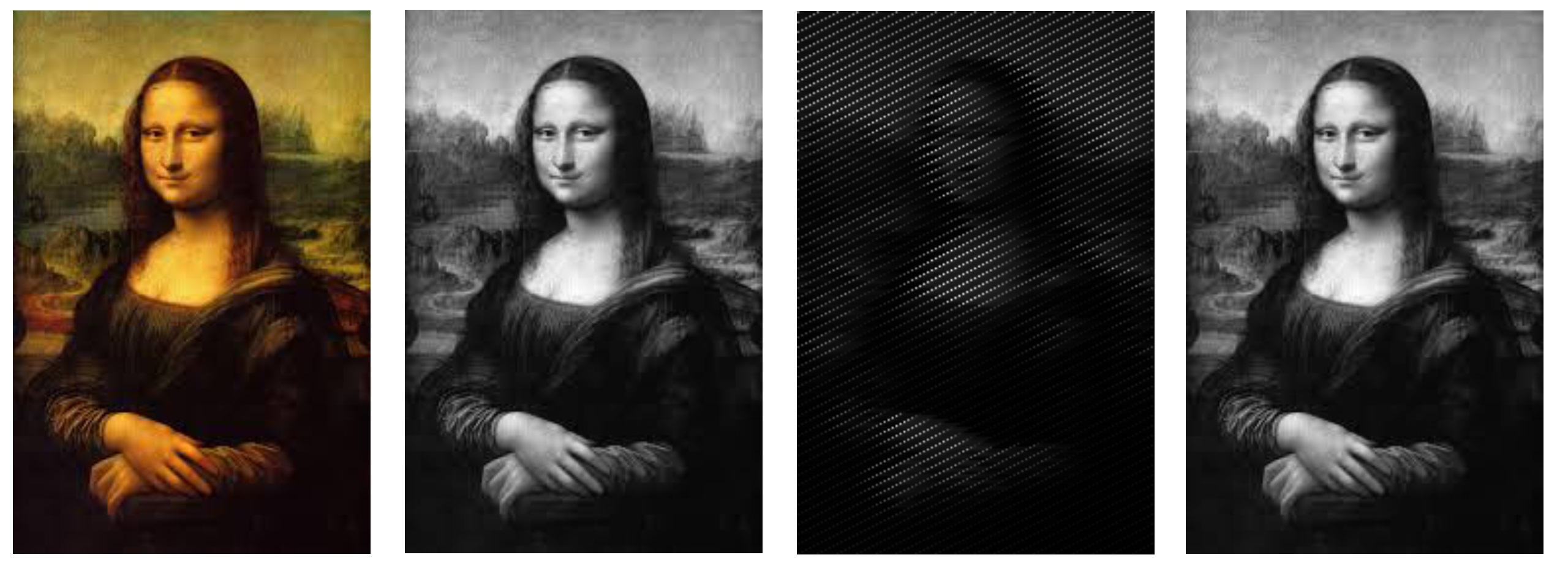}
\caption{Left: The original Mona Lisa. Middle Left: The Mona Lisa in gray
scale. Middle Right: The DTT using the sequence $\protect\gamma _{k}=1/(k+1)$%
, $k=0,1,2,...,9$. Right: The DTTMI recovered image.}
\label{MonaLisa1}
\end{figure}

The DTT can also be used to provide filtering of an image. Filters attempt
to exemplify certain features of a given image or modify images by altering
pixel values based on their neighbors. For example, filtering can be used to
achieve effects such as noise reduction (remove excess variability), edge
detection (boundary identification), bump maps (grayscale texture technique
used in 3d computer graphics to emulate small-scale surface details) and
image enhancement (improve the visual quality for better interpretation). We
illustrate how the DTT can be used to realize some of these concepts.

Recall the DTT form of Equation (\ref{DTTfdrivengamma})$,$ where for two
sequences $f_{n},\gamma _{n}:$ $\mathbb{N}\rightarrow \mathbb{M},$ with $%
\mathbf{f}=(f_{0},f_{1},...)$ and $\mathbf{\gamma }$ $=$ $(\gamma
_{0},\gamma _{1},...),$ $\gamma _{n}\neq 0,$ $\forall n\in B_{k}\in \mathcal{%
B}(\mathbb{N}),$ the DTT of $\mathbf{f}$ driven by the sequence $\mathbf{%
\gamma },$ is given by
\begin{equation*}
g_{k}(t)=\mathcal{T}_{\gamma _{k},B}^{f_{n}}(t)=\frac{1}{t}\sum\limits_{n\in
B_{k}}f_{n}\frac{\gamma _{k}^{n}}{t^{n}},
\end{equation*}%
$\forall k\in \mathbb{N},$ and take a real $t\neq 0.$ Here $\mathbf{f}$
plays the role of a vectorized image and $\mathbf{\gamma }$\ the driving
sequence describing a desired operation on the pixels $\mathbf{f}$, while
the created sequence $\mathbf{g}(t)=(g_{0}(t),g_{1}(t),...),$ corresponds to
the filtered image. The set $B_{k}$ plays the role of a neighborhood system,
and indicates which pixels of $\mathbf{f}$ (weighed by the $k^{th}$ element
of the driving sequence $\gamma _{k}),$ are used in order to create the new
filtered image $\mathbf{g}(t)$. Note that $B_{k}$ depends on the specific
pixel $k$ of the resulting image $\mathbf{g}(t)$.

Depending on the application context, we may desire coefficients that depend
on both the output index $k$ and input index $n$, and therefore we consider
a matrix-valued generalization of the original DTT, abbreviated as MVDTT. In
particular, define a two-index kernel by%
\begin{equation*}
K_{k,n}(t):=a_{k,n}\frac{\gamma _{k}^{n}}{t^{n+1}},
\end{equation*}%
such that%
\begin{equation*}
g_{k}(t)=\sum_{n\in B_{k}}K_{k,n}(t)f_{n}=\sum_{n\in B_{k}}f_{n}a_{k,n}\frac{%
\gamma _{k}^{n}}{t^{n+1}},
\end{equation*}%
and as a result, this generalization of the DTT is highly amenable to
changes, since we can choose the weights $a_{k,n}$ appropriately depending
on the application at hand.

Since images are finite, say of dimension $n_{1}\times n_{2},$ we will
assume that $\mathbf{f}$ is a fixed length vector of size $N\times 1$, with $%
N=n_{1}n_{2}$. In what follows, we will consider the standard 8-neighborhood
system, where the $(i,j)$-pixel value of an image is affected by its
neighbors. In particular, the indices for the neighborhood about the $(i,j)$%
-pixel, which corresponds to the $k=jn_{1}+i$ element of the vectorized
image, are given by
\begin{equation*}
\begin{tabular}{|l|l|l|}
\hline
$(i-1,j-1)$ & $(i-1,j)$ & $(i-1,j+1)$ \\ \hline
$(i,j-1)$ & $(i,j)$ & $(i,j+1)$ \\ \hline
$(i+1,j-1)$ & $(i+1,j)$ & $(i+1,j+1)$ \\ \hline
\end{tabular}%
\mapsto
\begin{tabular}{|l|l|l|}
\hline
$k-n_{1}-1$ & $k-1$ & $k+n_{1}-1$ \\ \hline
$k-n_{1}$ & $k$ & $k+n_{1}$ \\ \hline
$k-n_{1}+1$ & $k+1$ & $k+n_{1}+1$ \\ \hline
\end{tabular}%
.
\end{equation*}%
As a result, in order to apply a filter in the neighborhood of the $(i,j)$%
-pixel we select the set%
\begin{eqnarray*}
B_{k} &=&\{k-n_{1}-1,k-1,k+n_{1}-1,k-n_{1},k,k+n_{1}, \\
&&k-n_{1}+1,k+1,k+n_{1}+1\},
\end{eqnarray*}%
for all $k=1,2,...,N$, with appropriate changes at the boundaries of the
image. Thus the $k^{th}$ pixel of the filtered image is created using%
\begin{eqnarray*}
g_{k}(t) &=&\frac{1}{t}(f_{k-n_{1}-1}\frac{\left( \gamma _{k}\right)
^{k-n_{1}-1}}{t^{k-n_{1}-1}}+f_{k-1}\frac{\left( \gamma _{k}\right) ^{k-1}}{%
t^{k-1}}+f_{k+n_{1}-1}\frac{\left( \gamma _{k}\right) ^{k+n_{1}-1}}{%
t^{k+n_{1}-1}} \\
&&+f_{k-n_{1}}\frac{\left( \gamma _{k}\right) ^{k-n_{1}}}{t^{k-n_{1}}}+f_{k}%
\frac{\left( \gamma _{k}\right) ^{k}}{t^{k}}+f_{k+n_{1}}\frac{\left( \gamma
_{k}\right) ^{k+n_{1}}}{t^{k+n_{1}}} \\
&&+f_{k-n_{1}+1}\frac{\left( \gamma _{k}\right) ^{k-n_{1}+1}}{t^{k-n_{1}+1}}%
+f_{k+1}\frac{\left( \gamma _{k}\right) ^{k+1}}{t^{k+1}}+f_{k+n_{1}+1}\frac{%
\left( \gamma _{k}\right) ^{k+n_{1}+1}}{t^{k+n_{1}+1}}),
\end{eqnarray*}%
so that the kernel $K_{k,n}(t)=a_{k,n}\frac{\gamma _{k}^{n}}{t^{n+1}}$, in
this case has the special form $a_{k,n}=1,$ i.e., the weights $a_{k,n}$ are
independent of $k$ and $n$.

Now assume we wish to apply specific, fixed weights to the neighbors of a
pixel, say the matrix%
\begin{equation*}
\mathbf{A}=\left[
\begin{tabular}{lll}
$a_{11}$ & $a_{12}$ & $a_{13}$ \\
$a_{21}$ & $a_{22}$ & $a_{23}$ \\
$a_{31}$ & $a_{32}$ & $a_{33}$%
\end{tabular}%
\right] ,
\end{equation*}%
with $a_{ij}\in \Re ,$ $i,j=1,2,3,$ such that%
\begin{eqnarray*}
g_{k}(t) &=&\frac{1}{t}(f_{k-n_{1}-1}a_{11}\frac{\left( \gamma _{k}\right)
^{k-n_{1}-1}}{t^{k-n_{1}-1}}+f_{k-1}a_{12}\frac{\left( \gamma _{k}\right)
^{k-1}}{t^{k-1}}+f_{k+n_{1}-1}a_{13}\frac{\left( \gamma _{k}\right)
^{k+n_{1}-1}}{t^{k+n_{1}-1}} \\
&&+f_{k-n_{1}}a_{21}\frac{\left( \gamma _{k}\right) ^{k-n_{1}}}{t^{k-n_{1}}}%
+f_{k}a_{22}\frac{\left( \gamma _{k}\right) ^{k}}{t^{k}}+f_{k+n_{1}}a_{23}%
\frac{\left( \gamma _{k}\right) ^{k+n_{1}}}{t^{k+n_{1}}} \\
&&+f_{k-n_{1}+1}a_{31}\frac{\left( \gamma _{k}\right) ^{k-n_{1}+1}}{%
t^{k-n_{1}+1}}+f_{k+1}a_{32}\frac{\left( \gamma _{k}\right) ^{k+1}}{t^{k+1}}%
+f_{k+n_{1}+1}a_{33}\frac{\left( \gamma _{k}\right) ^{k+n_{1}+1}}{%
t^{k+n_{1}+1}}).
\end{eqnarray*}%
Owing to the flexibility offered by the MVDTT, this construction can be
accomplished by any driving sequence $\mathbf{\gamma }$ and choice of matrix
$\mathbf{A}$. Well known filters used for edge detection, such as the Sobel,
Scharr and Prewitt filters, are now easily implemented by the MVDTT using
the matrices%
\begin{equation*}
\mathbf{A}_{Sobel}=\left[
\begin{tabular}{lll}
$-1$ & $0$ & $1$ \\
$-2$ & $0$ & $2$ \\
$-1$ & $0$ & $1$%
\end{tabular}%
\right] ,\text{ }\mathbf{A}_{Scharr}=\left[
\begin{tabular}{lll}
$-3$ & $0$ & $3$ \\
$-10$ & $0$ & $10$ \\
$-3$ & $0$ & $3$%
\end{tabular}%
\right] ,\text{ and }\mathbf{A}_{Prewitt}=\left[
\begin{tabular}{lll}
$-1$ & $0$ & $1$ \\
$-1$ & $0$ & $1$ \\
$-1$ & $0$ & $1$%
\end{tabular}%
\right] ,
\end{equation*}%
respectively.

In Figure \ref{MonaLisa2} we present various MVDTT image processing results
for the Mona Lisa. The first row presents the results of the Sobel, Scharr
and Prewitt edge detection filters, along with the filtered image based on
the matrix $\mathbf{A}_{bump}=\left[
\begin{tabular}{rrr}
$-10$ & $0$ & $10$ \\
$0$ & $1$ & $0$ \\
$-10$ & $0$ & $10$%
\end{tabular}%
\right] $. The second row illustrates the blended images with pixels $%
b_{(i,j)}$ based on the original gray scale Mona Lisa with pixels $o_{(i,j)}$
and the filtered images with pixels $f_{(i,j)}$, using a weight of $w=0.7$,
i.e., the $(i,j)$-pixel in the blended image is created as the weighted
average $b_{(i,j)}=wo_{(i,j)}+(1-w)f_{(i,j)}.$ Note that this technique is
used to emulate a 3d perception to the original image, including surface
details such as bumps, wrinkles, or scratches.

\begin{figure}[t]
\centering \includegraphics[width=0.995\textwidth]{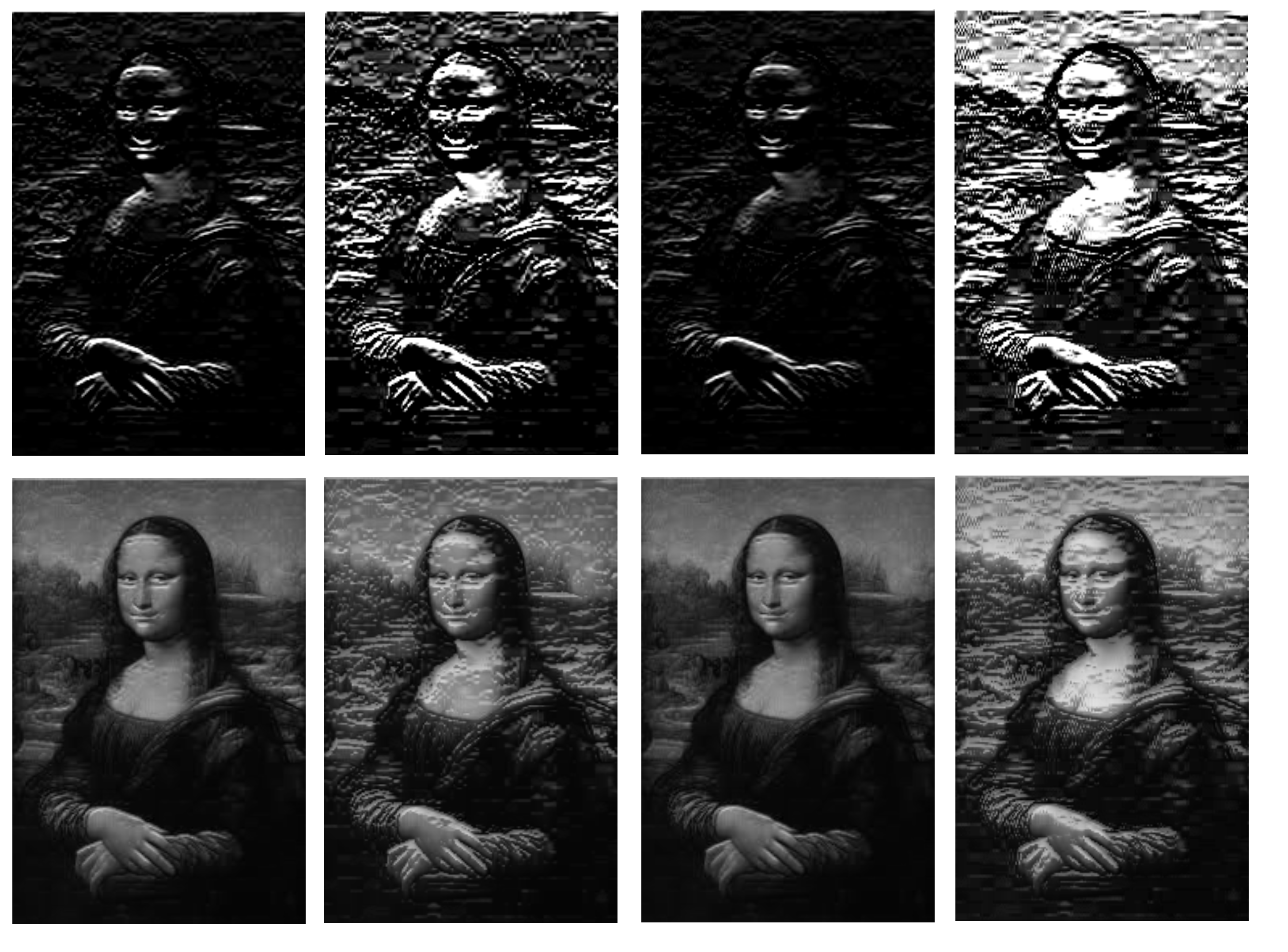}
\caption{First Row: Applications of the Sobel, Scharr and Prewitt edge
detection filters, along with the filtered image based on the matrix $%
\mathbf{A}_{bump}$.\newline
Second Row: Blended images using the original image and the filtered images
above. See the text for more details.}
\label{MonaLisa2}
\end{figure}

We illustrate MVDTT noise, blurring and smooth effects in Figure \ref%
{MonaLisa3}. Noise is added using a matrix that depends on the pixel
position, i.e., the matrix $\mathbf{A}^{(k)}$\ with zero elements except for
$a_{22}^{(k)}=1+X_{k},$ where $X_{k}\thicksim \mathcal{N}(0,\sigma ^{2}),$ $%
k=1,2,...,N$, a Gaussian error for a small $\sigma >0$ (in this example $%
\sigma =0.5$). Smoothing is accomplished by averaging all pixel values in
the neighborhood of the $(i,j)$-pixel using the matrix $\mathbf{A}$\ with $%
a_{ij}=1,$ $i,j=1,2,3$.

\begin{figure}[t]
\centering \includegraphics[width=0.995\textwidth]{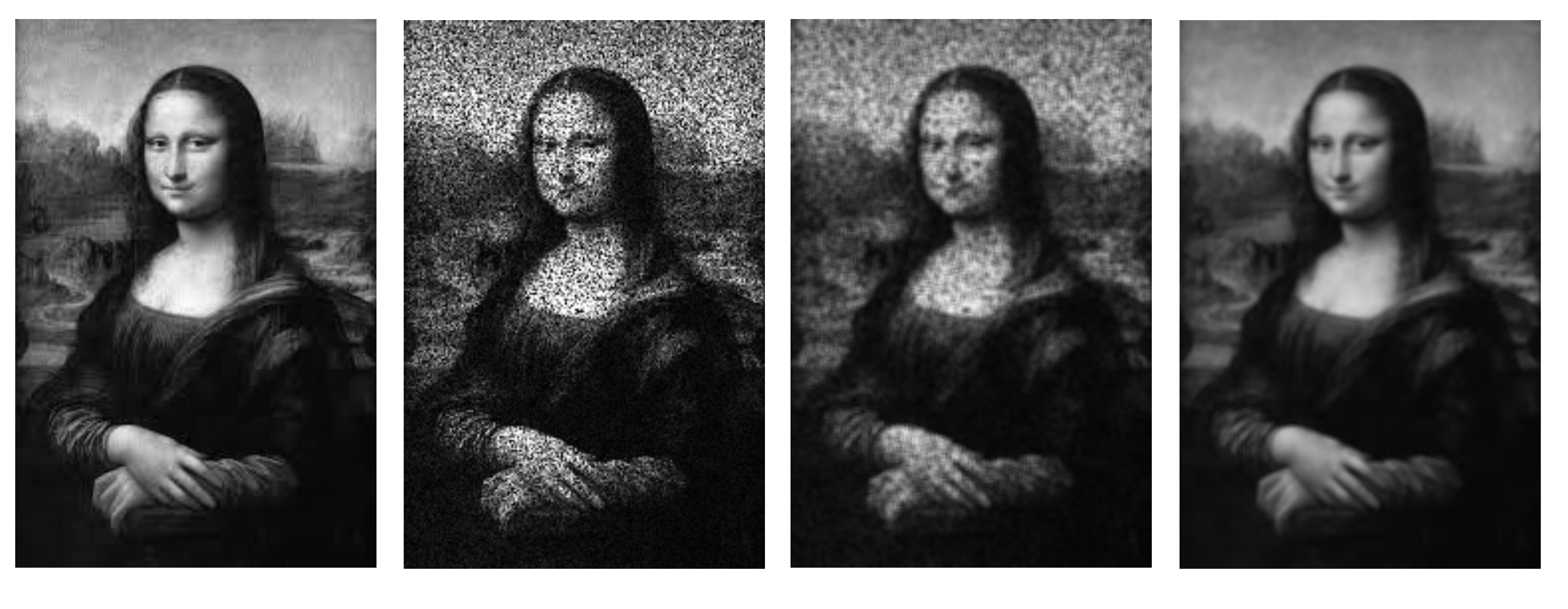}
\caption{Left: The original Mona Lisa. Middle Left: Distorted image using
Gaussian noise with mean 0 and standard deviation $\protect\sigma =0.5$.
Middle Right: Smoothing of the distorted image. Right: Smoothing of the
original Mona Lisa.}
\label{MonaLisa3}
\end{figure}

\section{Concluding Remarks}

Motivated by the Taylor measure, we introduced and studied properties of a
new integral, that emerges as a general framework that contains the finite
DFT as a particular choice of Taylor weights and nodes. The integral
provides a unifying framework, containing as special cases many important
concepts in mathematics.

We have only presented but a few applications of the DTT. Since the DFT has
many applications in all of the mathematical sciences, we will explore in
the future applying the DTT in those areas and investigate how we can
generalize those applications via the new method. Theoretical and
computational results are also being considered for the continuous version
of the TI\ and DTT.

Whereas the DFT takes advantage of a periodic $\mathbf{\gamma },$ we did not
discuss this case in this introductory paper. Assuming periodic driving
sequences will lead to the generalization of important applications, such as
time series analysis, signal processing and wavelets. In addition, we will
investigate driving sequences $\mathbf{\gamma }$ that will lead to fast DTT
versions, similarly to the fast Fourier transform. These results are of
great interest, and will be presented elsewhere.

\section*{Statements and Declarations}

\begin{itemize}
\item Conflict of Interest: The author declares no conflict of interest
regarding this paper.

\item Funding statement: This research received no external funding.

\item Acknowledgement: Not applicable.
\end{itemize}

\bibliographystyle{plainnat}
\bibliography{TaylorIntegral}

\end{document}